\documentclass[11pt,reqno]{amsart}
\usepackage[top=1in, bottom=1in, left=1.1in, right=1.1in]{geometry}
\usepackage{color}

\usepackage{amssymb}
\usepackage{amsmath}
\usepackage{amsfonts}
\usepackage{geometry}
\usepackage{graphicx}
\usepackage{mathrsfs,amssymb}

\usepackage{hyperref} 
\usepackage{cleveref}

\usepackage{enumerate}

\newtheorem{theorem}{Theorem}[section]
\newtheorem{proposition}[theorem]{Proposition}

\newtheorem{lemma}[theorem]{Lemma}
\newtheorem{corollary}[theorem]{Corollary}

\theoremstyle{definition}

\theoremstyle{remark}
\newtheorem{remark}[theorem]{Remark}

\numberwithin{equation}{section}

\renewcommand{\Re}{\operatorname{Re}}

\newcommand{\rg}{\operatorname{rg}}

\newcommand{\R}{\mathbb{R}}
\newcommand{\N}{\mathbb{N}}
\newcommand{\C}{\mathbb{C}}

\newcommand{\B}{\mathbb{B}}
\newcommand{\la}{\lambda}

\newcommand{\mc}[1]{\mathcal{#1}}

\title[]{Stable singularity formation for the Keller-Segel system in three dimensions}

\author{Irfan Glogi\'c}
\address{Department of Mathematics, University of Vienna, Oskar-Morgenstern-Platz 1, 1090 Vienna, Austria}
\email{irfan.glogic@univie.ac.at}
\author{Birgit Sch\"orkhuber}
\address{Leopold-Franzens Universit\"at Innsbruck, Institut f\"ur Mathematik, Technikerstraße 13, 6020 Innsbruck, Austria}
\email{Birgit.Schoerkhuber@uibk.ac.at}
\thanks{Irfan Glogi\'c is supported by the Austrian Science Fund FWF, Projects P 30076 and P 34378.}

\usepackage{kantlipsum}

\begin{document}
	\maketitle
\begin{abstract}
We consider the parabolic-elliptic Keller-Segel system in dimensions $d \geq 3$, which is the mass supercritical case. This system is known to exhibit rich dynamical behavior including singularity formation via self-similar solutions. An explicit example has been found more than two decades ago by Brenner et al.~\cite{BCKSV99}, and is conjectured to be nonlinearly radially stable. We prove this conjecture for $d=3$. Our approach consists of reformulating the problem in similarity variables and studying the Cauchy evolution in intersection Sobolev spaces via semigroup theory methods. To solve the underlying spectral problem, we crucially rely on a technique we recently developed in \cite{GloSch20}. To our knowledge, this provides the first result on stable self-similar blowup for the Keller-Segel system. Furthermore, the extension of our result to any higher dimension is straightforward.  We point out that our approach is general and robust, and can therefore be applied to a wide class of parabolic models. 
\end{abstract}

\section{Introduction}
\noindent We consider the following system of equations
\begin{equation}\label{Eq:KS}
	\begin{cases}
		\partial_t u(t,x)= \Delta u(t,x) + \nabla \! \cdot \! (u(t,x) \nabla v(t,x)), \\
		\Delta v(t,x) = u(t,x),
	\end{cases}
\end{equation}
equipped with an initial condition $u(0,\cdot)=u_0$, for $u, v : [0,T) \times \R^d \to \R$ and some $T>0$.
This model is frequently referred to as the parabolic-elliptic Keller-Segel system, named after the authors of \cite{Keller_Segel_1970}, who introduced a system of  coupled parabolic equations to describe chemotactic aggregation phenomena in biology. The parabolic-elliptic version \eqref{Eq:KS} was derived later by J\"ager and Luckhaus \cite{Jaeger_Luckhaus_1992}. System \eqref{Eq:KS} arises also as a simplified model for self-gravitating matter in stellar dynamics, with $u$ representing the gas density and $v$ the corresponding gravitational potential, see e.g. \cite{Wolansky1992, Ascibar_et_al_2013}. 

The equation for $v$ in \eqref{Eq:KS} can be solved explicitly in terms of $u$, which reduces the system to a
single (non-local) parabolic equation 
\begin{equation}\label{Eq:KS_single}
		\partial_t u(t,x)= \Delta u(t,x) + u(t,x)^2 + \nabla v_u(t,x) \nabla  u(t,x),
\end{equation} 
where $v_u = G \ast u$, with $G$ denoting the fundamental solution of the Laplace equation. This equation is invariant under the scaling transformation $u  \mapsto u_{\lambda}$, 
\[ u_{\lambda}(t,x) := \lambda^{-2} u(t/\lambda^2, x/\lambda), \quad \lambda >0. \] 
Furthermore, assuming sufficient decay of $u$ at infinity, the total mass 
\[ \mathcal  M(u)(t) = \int_{\mathbb R^d} u(t,x) dx, \]
is conserved. Since $\mathcal M(u_{\lambda}) = \lambda^{d-2} \mathcal M(u)(\cdot /\lambda^2)$, the model is mass critical for $d=2$ and mass supercritical for $d \geq 3$. 

It is well known that Eq.~\eqref{Eq:KS} admits finite-time blowup solutions in all space dimensions $d \geq 2$, for which in particular
\[ \lim_{t \to T^{-}}  \|u(t,\cdot) \|_{L^{\infty}(\R^d)} = \infty,\]
for some $T >0$. This is natural in view of the phenomena that the model is supposed to describe, and there is a strong interest in understanding the structure of singularities. Consequently, there is a huge body of literature addressing this question for~\eqref{Eq:KS} and variants thereof, for a review see e.g.~\cite{Horstmann1,Horstmann2}. 

Being the natural setting for biological applications, a lot of attention has centred around the  mass critical case  $d=2$. There, the $L^1$-norm of the stationary ground state solution $Q$, defined in \eqref{KS_2d_generic} below, represents the threshold for singularity formation, see e.g. \cite{BDP2006, BCM2008, GM2018, DdPDMW2019}. Particular solutions that blow up in finite time via dynamical rescaling of $Q$, 
\begin{align}\label{KS_2d_generic}
u(t,x)  \sim \frac{1}{\lambda(t)^2} Q\left (\frac{x}{\lambda(t)} \right ), \quad  Q(x) = \frac{8}{(1+|x|^2)^2}
\end{align}
with $\lambda(t) \to 0$ for $t \to T^-$, have been constructed for different blowup rates $\la$, see \cite{Collot_Ghoul_Masmoudi_Nguyen_2022_2d,  RS2014,  Velzaquez2002, HV1996}. In particular, in \cite{Collot_Ghoul_Masmoudi_Nguyen_2022_2d} it is shown that the blowup solution corresponding to 
\begin{equation}\label{Def:Blowup_rate}
	\lambda(t) = \kappa \sqrt{T-t} e^{-\sqrt{\frac{1}{2}| \log (T-t)|}}
\end{equation}
 for certain explicit constant  $\kappa > 0$, is stable outside of radial symmetry. In addition to this, Mizogouchi \cite{Mizoguchi2022} recently proved that for solutions with non-negative and radial initial data, \eqref{KS_2d_generic}-\eqref{Def:Blowup_rate} describes the universal blowup mechanism.

In comparison, the dynamics in the supercritical case $d \geq 3$ are more complex; in particular, multiple blowup profiles are known to exist. In a recent work, Collot, Ghoul, Masmoudi and Nguyen \cite{Collot_Ghoul_Masmoudi_Nguyen_preprint_2021} proved for all $d \geq 3$ the existence of a blowup solution that concentrates in a thin layer outside the origin and implodes towards the center. Other known examples of singular behavior are provided by self-similar solutions, which are proven to exist in all dimensions $d \geq 3$, see \cite{Herrero_Medina_Velazquez_1998, BCKSV99, Senba_2005}. A particular example was found  in closed form in  \cite{BCKSV99}, and is given by
	\begin{equation}\label{Eq:GroundState}
		u_T(t,x):=\frac{1}{T-t} U \left( \frac{x}{\sqrt{T-t}} \right) \quad \text{where} \quad U(x)=\frac{4(d-2)(2d+|x|^2)}{(2(d-2)+|x|^2)^2}.
	\end{equation}	 

\subsection{The main result}
To understand the role of the solution \eqref{Eq:GroundState} for generic evolutions of \eqref{Eq:KS}, the authors of \cite{BCKSV99} performed numerical experiments and conjectured as a consequence that $u_T$ is nonlinearly radially stable. In spite of a number of results on the nature of blowup, this conjecture has remained open for more than two decades now.
In the main result of this paper we prove this conjecture for $d=3$. More precisely, we show that there is an open set of radial initial initial data around $u_1(0,\cdot)=U$ for which the Cauchy evolution of \eqref{Eq:KS} forms a singularity in finite time $T > 0$ by converging to $u_T$, i.e., to the profile $U$ after self-similar rescaling. The formal statement is as follows.
\begin{theorem}\label{Thm:Main}
		Let $d = 3$. There exists $\varepsilon>0$ such that for any initial datum
	\begin{equation}\label{Eq:Pert}
		u(0,\cdot)=U +\varphi_0,
	\end{equation}
	where $\varphi_0$ is a radial Schwartz function for which
	\begin{equation*}
		\| \varphi_0  \|_{{H}^3(\R^3)} < \varepsilon,
	\end{equation*}
	there exists $T>0$ and a classical solution $u \in C^\infty([0,T) \times \R^3)$ to \eqref{Eq:KS}, which blows up at the origin as $t \rightarrow T^{-}$. Furthermore, the following profile decomposition holds
	\begin{equation*}
		u(t,x)=\frac{1}{T-t}\left[ U\left(\frac{x}{\sqrt{T-t}}\right) + \varphi \left(t, \frac{x}{\sqrt{T-t}} \right) \right],
	\end{equation*}
	where
$
	\| \varphi (t,\cdot) \|_{{H}^3(\R^3)} \rightarrow 0
$
	as $t \rightarrow T^-$.
\end{theorem}

\begin{remark}
As it will be apparent from the proof, the extension of this result to any higher dimension is straightforward. This involves developing the analogous well-posedness theory and solving the underlying spectral problem for a particular choice of $d \geq 4$. We therefore restrict ourselves to the lowest dimension, and the physically most relevant case, $d=3$. 
\end{remark}

\begin{remark}
	Due to the embedding $H^3(\R^3) \hookrightarrow L^\infty(\R^3)$, 
	the conclusion of the theorem implies  that the evolution of the perturbation \eqref{Eq:Pert}, when dynamically self-similarly rescaled, converges back
	to $U$ in $L^\infty(\R^3)$. In other words,
	\begin{equation*}
		(T-t)\, u (t,\sqrt{T-t}\, \cdot) \rightarrow U
	\end{equation*}
uniformly on $\R^3$ as $t \rightarrow T^-$. 
\end{remark}

\subsection{Related results for $d \geq 3$}\label{Sec:Known_results}
There are many works that treat the system \eqref{Eq:KS} in higher dimensions. Here we give a short and noninclusive overview of some of the important developments. 

Local existence and uniqueness of radial solutions for \eqref{Eq:KS} holds in $L^{\infty}(\R^d)$ as well as in other function spaces, see e.g.~\cite{Giga_Mizoguchi_Senba_2011, Biler_Book2020}. Concerning global existence, various criteria are given in terms of critical (i.e.~scaling invariant) norms. For example, it is known that initial data of small $L^{d/2}(\R^d)$-norm lead to global (weak) solutions \cite{CPZ2004}. This result was later extended by Calvez, Corrias, and Ebde \cite{CPE2012} to all data of norm less than a certain constant coming from the Gagliardo-Nierenberg inequality. For results in terms of the critical Morey norms, see e.g.~\cite{BilKarPil19,Biler_Book2020}. 
Concerning the existence of finite time blowup, the aforementioned works \cite{CPZ2004,CPE2012} give sufficient conditions in terms of the size of the second moment of the initial data. For an earlier result of that type see the work of Nagai \cite{Nagai1995}. For other, more recent results see \cite{BilKarZie15,Ogawa_Wakui2016,SW2019,BilZie19,Biler_Book2020,Naito2021}.  We point out, however, that in contrast to the $d=2$ case, for $d \geq 3$ still no simple  characterization of threshold for blowup in terms of a critical norm is known.

Concerning the structure of singularities, not much is known. It is straightforward to conclude that blowup solutions of \eqref{Eq:KS} satisfy $\liminf_{t \to T^{-}} \,(T-t) \| u(t,\cdot) \|_{L^{\infty}(\R^d)} >0$, see e.g.~\cite{NaiSen12}. Accordingly, singular solutions are classified as \textit{type I} if  
\[ \limsup_{t \to T^{-}} \,(T-t) \| u(t,\cdot) \|_{L^{\infty}(\R^d)} < \infty \]
and \textit{type II} otherwise.   
The first formal construction of type II blowup was performed by Herrero, Medina and Vel\'{a}zquez for $d=3$ in \cite{Herrero_Medina_Velazquez_1997}; the singularity they construct consists of a smoothed-out shock wave concentrated in a ring that collapses into a Dirac mass the origin. This blowup mechanism was later observed numerically in \cite{BCKSV99} for higher dimensions as well, and is furthermore conjectured to be radially stable. A rigorous construction of this solution for all $d \geq 3$ came only recently in the work of Collot, Ghoul, Masmoudi and Nguyen \cite{Collot_Ghoul_Masmoudi_Nguyen_preprint_2021}, who also prove its radial stability.
In contrast to these results, if the initial profile of a blowup solution is radially non-increasing and of finite mass then the limiting spatial profile is very much unlike the Dirac mass, since it satisfies
\begin{equation*}
	C_1 |x|^{-2} \leq \lim_{t \rightarrow T^-} u(t,x) \leq C_2|x|^{-2}
\end{equation*}
near the origin, as proven by Souplet and Winkler \cite{SW2019}. This is in particular the case for self-similar solutions, for which $\lim_{t \rightarrow T^-} u(t,x) = C|x|^{-2}$ at the blowup time.
To add to the importance of self-similar solutions for understanding the structure of singularities, Giga, Mizoguchi and Senba   \cite{Giga_Mizoguchi_Senba_2011} showed that any radial, non-negative type I blowup solution of \eqref{Eq:KS} is asymptotically self-similar. For $3 \leq d \leq 9$, it is known that there are infinitely many similarity profiles, while for $d \geq 10$ there is at least one, see \cite{Senba_2005}. However, a full classification of the set of self-similar (blowup) solutions, even in radial symmetry, is not available so far.

Finally, we note that in the two-dimensional case, any blowup solution is necessarily of type II, see e.g.~ \cite{NaiSuz08}. In the mass subcritical case $d=1$ blowup has been excluded in \cite{ChiPer81}.

\subsection{Outline of the proof of the main result}
Since we assume radial symmetry, i.e.~$u(t,x) = \tilde u(t,|x|)$, we first reformulate \eqref{Eq:KS} in terms of the reduced mass  
\begin{equation*}
	\tilde w(t,r):=\frac{1}{2r^2} \int_{0}^{r} \tilde u(t,s)s^{2}ds.
\end{equation*}
The advantage of this change of variable lies in the fact that it reduces the system \eqref{Eq:KS} to a single local semilinear heat equation on $\mathbb R^{5}$ for $w(t,x) := \tilde w(t,|x|)$ 
\begin{equation}\label{Eq:w}
  \begin{cases}
	~	\partial_t w  -  \Delta w = \Lambda w^2 + 6 w^2,\\
	~ w(0,\cdot)=w_0,
  \end{cases}
\end{equation}
where $\Lambda f(x):=x \cdot \nabla f(x)$, and the initial datum is radial $w_0=\tilde{w}(0,|\cdot|)$. 
Additionally, the self-similar solution \eqref{Eq:GroundState} turns into
\begin{equation*}
	w_T(t,x):=\frac{1}{T-t}\phi\left( \frac{x}{\sqrt{T-t}}\right) \quad \text{where} \quad \phi(x)=\frac{2}{2+|x|^2}.
\end{equation*}
The bulk of our proof consists of showing stability of $w_T$. Then, by using the equivalence of norms of $u$ and $w$ we turn the obtained stability result into Theorem \ref{Thm:Main}.
In Section \ref{Sec:Sim_Var},
as is customary in the study of self-similar solutions, we pass to similarity variables
\[\tau = \tau(t) := \log\left(\frac{T}{T-t}\right), \quad \xi = \xi(t,x) :=\frac{x}{\sqrt{T-t}}. \]
We remark that the application of similarity variables in the study of blowup for nonlinear parabolic equations goes back to 1980's and the early works of Giga and Kohn \cite{GigKoh85,GigKoh87,GigKoh89}.
Note that the strip $[0,T)\times \R^5$ is mapped via $(t,x) \mapsto (\tau,\xi)$ into the half-space $[0,\infty) \times \R^5$. Furthermore, by rescaling the dependent variable $(T-t)w(t,x)=:\Psi(\tau,\xi)$, the solution $w_T$ becomes $\tau$-independent, $\xi \mapsto \phi(\xi)$.  
Consequently, the problem of stability of finite time blowup via $w_T$ is turned into the problem of the asymptotic stability of the static solution $\phi$. To study evolutions near $\phi$, we consider the pertubation ansatz $\Psi(\tau,\cdot) = \phi + \psi(\tau)$, which yields the central evolution equation of the paper
\begin{equation}\label{Eq:Pert_intro}
	\begin{cases}
		~	\partial_\tau \psi(\tau) =  L \psi(\tau) +  N\big(\psi(\tau)\big),\\
		~ \psi(0) =  T w_0(\sqrt{T}\cdot)-\phi.
	\end{cases}
\end{equation}
Here, the linear operator $L$ and the remaining nonlinearity $N$ are explicitly given as
\begin{equation}\label{Eq:L,N_intro}
	L  = \Delta  - \frac12 \Lambda  - 1	 + 2\Lambda(\phi \, \cdot)+ 12 \phi 
	\quad \text{and} \quad  N(f)= \Lambda f^2 +  6 f^2.
\end{equation} 
The next step is to establish a well-posedness theory for the Cauchy problem \eqref{Eq:Pert_intro}. To that end, in Section \ref{Sec:Funct} we introduce the principal function space of the paper
\begin{equation*}
	X^k := \dot H^1_{\text{rad}}(\mathbb R^5) \cap \dot H^k_{\text{rad}}(\mathbb R^5), \quad k \geq 3.
\end{equation*}

 To construct solutions to \eqref{Eq:Pert_intro} in $X^k$, we take up the abstract semigroup approach. 
  In Section \ref{Sec:Evol_Lin} we concentrate on the linear version of \eqref{Eq:Pert_intro}. First, we prove that $L$, being initially defined on test functions, is closable, and its closure $\mc L$ generates a strongly continuous semigroup $S(\tau)$ on $X^k$. Our proof is based on the Lumer-Phillips theorem, and it involves a delicate construction of global, radial and decaying solutions to the Poisson type equation $\mc {L}f=g$. Thereby we establish existence of the linear flow near $\phi$. This flow exhibits growth in general, due to the existence of the unstable eigenvalue $\la=1$ of the generator $\mc L$. This instability is not a genuine one though, as it arises naturally, due to the time translation invariance of the problem. By combining the analysis of the linear evolution in $X^k$  with a thorough spectral analysis of $\mc L$ in a suitably weighted $L^2$-space, we prove the existence of a (non-orthogonal)  rank-one projection $\mc P: X^k \rightarrow X^k$ relative to $\la=1$, such that the linear evolution in $X^k$ decays exponentially on $\ker \mc P$. This is expressed formally in the central  result of the linear theory, Theorem \ref{Thm:DecayX}.  
  
  The existence of $S(\tau)$ allows us to express the nonlinear equation \eqref{Eq:Pert_intro} in the integral from
  \begin{equation}\label{Eq:Int_intro}
  	\psi(\tau) = S(\tau)\psi(0) + \int_{0}^{\tau}S(\tau-s)N(\psi(s))ds.
  \end{equation}
  As is customary, we employ a fixed point argument to construct solutions to \eqref{Eq:Int_intro} in $X^k$. For this, we need a suitable Lipschitz continuity property of the nonlinear operator $N$ in $X^k$. However, the  space $X^k$ is not invariant under the action of $N$, due to the presence of the derivative nonlinearity, recall \eqref{Eq:L,N_intro}. Nevertheless, by exploiting the smoothing properties of $S(\tau)$, we show that the operator $f \mapsto S(\tau)N(f)$ is locally Lipschitz continuous in $X^k$, which will suffice for setting up a contraction scheme for \eqref{Eq:Int_intro}. The proofs of the smoothing properties of $S(\tau)$ and those of the accompanying nonlinear Lipschitz estimates comprise the main content of Section \ref{Sec:Nonlin}.
  
  With these technical results at hand, in Section \ref{Sec:Str_sol} we use a fixed point argument to construct for \eqref{Eq:Int_intro} global, exponentially decaying strong $X^4$-solutions for small data. To deal with the growth stemming from the presence of $\rg \mc P$ in the initial data, we employ a Lyapunov-Perron type argument, by means of which we also extract the blowup time $T$. In Section \ref{Sec:Upgrade_to_class}, we use regularity arguments to show that the constructed strong solutions are in fact classical. By translating the obtained result back to physical coordinates $(t,x)$, we get stability of $w_T$. Finally, by means of the equivalence of norms of $w$ and $u$, from this we derive Theorem \ref{Thm:Main}.

We finalize with a remark that our methods can straightforwardly be generalized so as to treat \eqref{Eq:w} in any higher dimension $n \geq 6$. This involves carrying out the analogous well-posedness theory for \eqref{Eq:Pert_intro} in the high-dimensional counterpart of the space $X^k$
\begin{equation*}
	X_n^k:= \dot{H}_{\text{rad}}^{\left\lfloor \frac{n}{2}\right\rfloor -1}(\R^n) \cap \dot{H}_{\text{rad}}^{k}(\R^n), \quad k \geq \left\lfloor \tfrac{n}{2}\right\rfloor + 1,
\end{equation*}
along with using the same techniques to treat the underlying spectral problem. 
\subsection{Notation and Conventions}
We write $\N$ for the natural numbers $\{1,2,3, \dots\}$, $\N_0 := \{0\} \cup \N$. Furthermore, $\R^+ := \{x \in \R: x >0\}$.
By $C_{c}^\infty(\R^d)$ we denote the space of smooth functions with compact support. In addition, we define
$ C_{\mathrm{rad	}}^{\infty}(\R^d) := \{ u \in C^{\infty}(\R^d): u \text{ is radial}\}$
and analogously $C^{\infty}_{c,\mathrm{rad	}}(\R^d)$. 
 Also, $\mc S_{\text{rad}}(\R^d)$ stands for  the space of radial Schwartz functions.
By $L^p(\Omega)$ for $\Omega \subseteq \R^d$, we denote the standard Lebesgue space. For the Fourier transform we use the following convention
\begin{equation}\label{Def:FourierT}
	\mc F_d(f)(\xi) = \frac{1}{(2 \pi)^{d/2}} \int_{\R^d} e^{- i \xi \cdot x} f(x) dx.
\end{equation}
For a closed linear operator $(L, \mc D(L))$, we denote by $\rho(L)$ the resolvent set, and for $\lambda \in \rho(L)$ we use the following convention for the resolvent operator
$R_{L}(\lambda):=(\lambda- L)^{-1}$.
The spectrum is defined as $\sigma(L) := \C \setminus \rho(L)$. The notation $a\lesssim b$ means $a\leq Cb$ for some $C>0$, and we write $a\simeq b$ if $a\lesssim b$ and $b \lesssim a$.  	
We use the common notation $\langle x \rangle := \sqrt{1+|x|^2}$ also known as the \textit{Japanese bracket}.

\section{Equation in similarity variables}\label{Sec:Sim_Var}

\noindent To restrict to radial solutions of \eqref{Eq:KS} we assume that $u(t,\cdot)=\tilde{u}(t,|\cdot|)$ and $u_0=\tilde{u}_0(|\cdot|)$. Then, we define the so-called reduced mass 
\begin{equation*}
	\tilde{w}(t,r):=\frac{1}{2r^d} \int_{0}^{r}\tilde u(t,s)s^{d-1}ds,
\end{equation*}
and let $n:=d+2$. The utility of the reduced mass is reflected in the fact that by means of the function $w(t,x):=\tilde{w}(t,|x|)$, the initial value problem for the system \eqref{Eq:KS} can be rewritten in the form of a single local semilinear heat equation in $w$
\begin{equation}\label{Eq:NLH}
  \begin{cases}
	~	\partial_t w - \Delta w = \Lambda w^2 + 2d w^2,\\
	~ w(0,\cdot)=w_0.
  \end{cases}
\end{equation}
Here $(t,x) \in [0,T) \times \R^{n}$, 
\begin{equation*}
	\Lambda f(x):=x \cdot \nabla f(x),
\end{equation*}
 and 
\begin{equation}
	 w_0=\tilde w_0(|\cdot|) \quad \text{for} \quad \tilde{w}_0(r)=	\frac{1}{2r^d} \int_{0}^{r}\tilde u_0(s) s^{d-1}ds.
\end{equation}
 Furthermore, the self-similar solution \eqref{Eq:GroundState} turns into
\begin{equation*}
	w_T(t,x):=\frac{1}{T-t}\phi_n\left( \frac{x}{\sqrt{T-t}}\right) \quad \text{where} \quad \phi_n(x)=\frac{2}{2(n-4)+|x|^2}.
\end{equation*}

\subsection{Similarity variables}
We pass to similarity variables
\begin{equation}\label{Def:Simil_var}
	\tau= \tau(t):=\ln \left(\frac{T}{T-t} \right),  \quad \xi=\xi(t,x):=\frac{x}{\sqrt{T-t}}, \quad T>0.
\end{equation}
Note that transformation \eqref{Def:Simil_var} maps the time slab $S_T:=[0,T) \times \R^n$ 
into the upper half-space $H_+:=[0,+\infty) \times \R^n$. Also, we define the rescaled dependent variable
\begin{equation}\label{Def:psi}
	\Psi(\tau,\xi):=(T-t)w(t,x)=Te^{-\tau}w(T-Te^{-\tau},\sqrt{T}e^{-\frac{\tau}{2}}\xi).
\end{equation}
Consequently, the evolution of $w$ inside $S_T$ corresponds to the evolution of $\Psi$ inside $H_+$. Furthermore, since
\begin{equation}\label{Eq:Diff_law}
	\partial_t = \frac{e^\tau}{T}\left(\partial_{\tau} + \frac{1}{2}\Lambda \right),  \quad \Delta_x = \frac{e^\tau}{T}\Delta_\xi,
\end{equation}
we get that the nonlinear heat equation \eqref{Eq:NLH} transforms into
\begin{equation}\label{Eq:NLH_sim_var}
	\left(\partial_\tau  - 	\Delta  + \frac{1}{2} \Lambda  + 1 \right)\Psi = \Lambda \Psi^2 + 2(n-2) \Psi^2,
\end{equation}
with the initial datum 
\begin{equation*}
	\Psi(0,\cdot)= Tw_0(\sqrt{T}\cdot).
\end{equation*}
For convenience, we denote
\begin{equation}\label{Def:L0}
		L_0  := \Delta  - \frac12 \Lambda  - 1.
\end{equation}
Now, we have that for all $n \geq 5$ the function
\begin{equation*}
	\phi_n(\xi)= \frac{2}{2(n-4)+|\xi|^2}
\end{equation*}
is a static solution to \eqref{Eq:NLH_sim_var}. To analyze stability properties of $\phi_n$, we study evolutions of initial data near $\phi_n$, and for that we consider the perturbation ansatz
\begin{equation*}
	\Psi(\tau,\cdot) = \phi_n + \psi(\tau).
\end{equation*} 
This leads to the central evolution equation of the paper
\begin{equation}\label{Eq:Perturbed}
	\partial_\tau \psi(\tau) = \big( L_0 + L' \big)\psi(\tau) +  N\big(\psi(\tau)\big), 
\end{equation}
where
\begin{equation}\label{Def:L'}
	 L'f =  2\Lambda(\phi_n f)+4  (n-2) \phi_n f 
	\quad \text{and} \quad  N(f)= \Lambda f^2 +  2(n-2) f^2.
\end{equation}
Furthermore, we write the initial datum as
\begin{equation}\label{Eq:Init_data_pert}
	\psi(0)= \Psi(0,\cdot) - \phi_n = \big(T w_0(\sqrt{T}\cdot)-w_0 \big) + v =: U(v,T)
\end{equation}
where, for convenience, we denoted
\begin{equation}\label{Def:InitCond_v}
	 v = w_0 - \phi_n.
\end{equation}
Now we fix $n=5$, and study the Cauchy problem \eqref{Eq:Perturbed}-\eqref{Eq:Init_data_pert}. For this we need a convenient functional setup.

\section{Functional setup}\label{Sec:Funct}

\noindent This section is devoted to defining the function spaces in which we study the Cauchy evolution of \eqref{Eq:Perturbed}. Furthermore, we gather the basic embedding properties that will be used later on.

\subsection{The space $X^k$}
Let $k \in \N$ and $f \in C^{\infty}_{c,\mathrm{rad	}}(\R^5)$. As usual, we  define the  Sobolev norm  $\| f \|_{\dot H^{k}(\R^5)} := \| |\cdot|^{k} \mc F f \|_{L^2(\R^5)}$ via the Fourier transform.
 Since we are concerned  with radial functions only, i.e.,  $f = \tilde f(|\cdot|)$, we straightforwardly get that
\[ \| f \|_{\dot H^{k}(\R^5)} = \| D^{k} f\|_{L^2(\R^5)},\]
with 
\begin{align*}
D^k f  := \begin{cases}  \Delta^{\frac{k}{2}}_{\mathrm{rad}} \tilde f(|\cdot|),  & \text{for \emph{k} even,} \\[3mm]
\big(\Delta^{\frac{k-1}{2}}_{\mathrm{rad}} \tilde f \big)'(|\cdot|),  &  \text{for \emph{k} odd,}
		\end{cases}					   
\end{align*}
where  $ \Delta_{\mathrm{rad}} := r^{-4} \partial_r ( r^{4} \partial_r)$ denotes the radial Laplace operator on $\R^5$. We also define an inner product on  $C^{\infty}_{c,\mathrm{rad	}}(\R^5)$ by
\begin{equation*}
	\langle f,g\rangle_{X^k} := \langle  D f, D g\rangle_{L^2(\R^5)} + \langle D^k f, D^k g\rangle_{L^2(\R^5)},
\end{equation*}
with corresponding norm $\|f \|_{X^k} := \sqrt{\langle f,f\rangle_{X^k}}$.
The central space of our analysis is the completion of $(C^{\infty}_{c,\mathrm{rad	}}(\R^5), \| \cdot \|_{X^k})$, and we denote it by $X^{k}$. Throughout the paper we frequently use the equivalence
\begin{align}
\| f \|^2_{X^k} = \| f\|^2_{\dot H^1(\R^5)} + \| f\|^2_{\dot H^k(\R^5)} \simeq \sum_{|\alpha| = 1} \| \partial^{\alpha} f \|^2_{L^2(\R^5)} + \sum_{|\alpha| = k} \| \partial^{\alpha} f \|^2_{L^2(\R^5)}.
\end{align}
Now, we list several properties of $X^k$ that will be used later on. First, a simple application of the Fourier transform yields the following interpolation inequality.

\begin{lemma}\label{Le:Interpolation}
Let $k \in \N$. Then 
\begin{align*}
\|\partial^\alpha f \|_{L^2(\R^5)} \lesssim \| f \|_{X^k}
\end{align*}
for all $f \in C^{\infty}_{c,\mathrm{rad	}}(\R^5)$, and all $\alpha \in \N_0^5$ with $1 \leq |\alpha|\leq k.$
\end{lemma}

Also, it is straightforward to see that the following result holds. 

\begin{lemma}\label{Le:Functions_in_Xk}
Let $k \in \N$, $k \geq 3$. Then 
\[ X^{k} \hookrightarrow W_{\mathrm{rad}}^{{k-3},\infty}(\R^5).\]
In particular, elements of $X^k$ can be identified with  functions in $C_\mathrm{rad	}^{k-3}(\R^5)$. Furthermore, $X^{k}$ is a Banach algebra, i.e.,  
\[ \|  f g\|_{X^k} \lesssim \|f \|_{X^k} \| g \|_{X^k} \]
for all $f,g \in X^{k}$. 
In addition, $X^{k_2}$ embeds continuously into $X^{k_1}$ for $ k_1 \leq k_2$; shortly
\begin{equation}\label{Eq:Embed_Xk}
	X^{k_2}\hookrightarrow X^{k_1}.
\end{equation}
\end{lemma} 
The following result follows from an elementary approximation argument.

\begin{lemma}\label{Lemma_Def_C}
Let $k \in \N$, $k \geq 3$. Define
\[ \mc C:= \{ f \in C_{\emph{rad}}^{\infty}(\R^5): \text{ for } \kappa \in \N_0, \, |D^{\kappa} f(x)| \lesssim \langle x \rangle^{-2-\kappa}  \}. \]
Then $\mc C \subset X^{k} $.
\end{lemma}

Throughout the paper, we frequently use the commutator relation
\begin{align}\label{Eq:Comm_Lambda_2}
\partial^{\alpha} \Lambda = \Lambda \partial^{\alpha} + k \partial^{\alpha},
\end{align}
for $\alpha \in \N_0$, $|\alpha| = k$, as well as its radial analogue
\begin{equation}\label{Comm_DLambda}
D^k \Lambda = \Lambda D^k + k 	D^k.
\end{equation}

\subsection{Weighted $L^2-$spaces}

For $x \in \R^5$, we set 
 \begin{equation*}
 	\sigma_0(x):= e^{-|x|^2/4} \quad \text{and} \quad \sigma(x):= \phi(x)^{-2} e^{-|x|^2/4}
 \end{equation*}
 where 
\[\phi(x):= \phi_5(x) = \frac{2}{2+|x|^2}.\]
We then define the following weighted $L^2$-spaces of radial functions,
\begin{equation*}
	\mc H_0 := \{ f \in L^2_{\sigma_0}(\R^5) : f\text{ is radial} \}, \quad \mc H := \{ f \in L^2_{\sigma}(\R^5) : f\text{ is radial} \},
\end{equation*}
with the underlying inner products
\begin{equation*}
\langle f ,g \rangle_{\mc H_0} = \int_{\R^5}  f(x) \overline{g(x)} \sigma_0(x) dx,  \quad \langle f ,g \rangle_{\mc H} = \int_{\R^5}  f(x) \overline{g(x)} \sigma(x) dx.
\end{equation*}
 We note that both $\mc H_0$ and $\mc H$ have $C^{\infty}_{c,\mathrm{rad}}(\R^5)$ as a dense subset. An immediate consequence of the exponential decay of the weight functions is the following result.

\begin{lemma}\label{Le:Embedding_X_H}
Let  $k \in \N, k \geq 3$. Then
\begin{align*}\label{EstEmbedding:L2sigma_X}
 \| f \|_{\mc H_0} \lesssim  \| f \|_{\mc H} \lesssim \|f \|_{L^{\infty}(\R^5)} \lesssim \|f \|_{X^k}  
\end{align*}
for every $f \in C^{\infty}_{c,\mathrm{rad}}(\R^5)$. Consequently, we have the following continuous embeddings 
\[ X^k \hookrightarrow   \mc H \hookrightarrow \mc H_0  . \]
\end{lemma} 

\subsection{Operators}
For convenience, we copy here the linear operators defined in \eqref{Def:L0} and \eqref{Def:L'} 
\begin{gather}\label{Def:Operators}
	L_0 f := \Delta f - \frac12 \Lambda f - f, \quad L' f:= 2\Lambda(\phi f)+12\phi f, \quad  N(f) := \Lambda f^2 +  6 f^2,
\end{gather}
and we formally let
\begin{align*}
	L:=L_0+L'.
\end{align*}
Note that by Lemma \ref{Lemma_Def_C}, $\phi \in X^{k}$ for any $k \in \N , k \geq 3$, and the same holds for $\Lambda \phi$.

\section{Linear theory}\label{Sec:Evol_Lin}
\noindent In this section we concentrate on the linear version of \eqref{Eq:Perturbed}, and show that it is well-posed in $X^k$ for $k \geq 3$. To accomplish this, we use the semigroup theory. Before we state the central result of the section, we make some technical preparations. First, we note that due to the underlying time-translation symmetry, the linear operator $L$ has a formal unstable eigenvalue, $\la=1$, with an explicit eigenfunction 
\begin{equation}\label{Def:nu}
	 \nu(x) :=  \phi(x) + \frac12 \Lambda \phi(x) = \frac{4}{(2+|x|^2)^2}.
\end{equation}
 According to Lemma \ref{Lemma_Def_C}, the function $\nu$ belongs to $X^k$, and this allows us to define a projection operator $\mc P : X^k \rightarrow X^k$ by
\begin{equation}\label{Def:P}
	\mc P f := \langle f,  g \rangle_{\mc H} ~  g, \quad \text{where} \quad g = \frac{\nu}{\| \nu \|_{\mc H}}.
\end{equation} 
This whole section is devoted to proving the following theorem, which, in short, states that the linear flow of \eqref{Eq:Perturbed} decays exponentially in time on the kernel of $\mc P$.

\begin{theorem}\label{Thm:DecayX}
Let $k \in \N$, $k \geq 3$. Then the operator $L: C^{\infty}_{c,\mathrm{rad	}}(\R^5)  \subset X^k \to X^k$ is closable, and its closure generates a strongly continuous semigroup  $(S(\tau))_{\tau \geq 0}$ of bounded operators on $X^k$. Furthermore, there exists $\omega_k \in (0,\frac{1}{4})$ such that
\begin{align}\label{Eq:S_decay}
\| S(\tau) (1- \mc P) f \|_{X^k}   \lesssim e^{-\omega_{k} \tau } \|(1- \mc P) f \|_{X^k} \quad \text{and} \quad S(\tau)\mc P f = e^{\tau} \mc P f,
\end{align}
for all $f \in X^k$ and all $\tau \geq 0$.
\end{theorem}
The mere fact that the closure of $L$ generates a semigroup on $X^k$ can be proved  by an application of the Lumer-Phillips Theorem, see Section  \ref{Evol:Xk}. However, determining the precise growth of the semigroup is highly non-trivial, in view of the non-self-adjoint nature of the problem and the lack of an abstract spectral mapping theorem that would apply to this situation. To get around this issue, we combine the analysis in $X^{k}$ with  the self-adjoint theory for $L$ in $\mc H$.

\subsection{The linearized evolution on $\mc H$}\label{Sec:Evol_H}
We equip $L_0$ and $L$ with domains  
\begin{align}\label{Def:Domain_L_L0}
\mc D(L_0)=\mc D(L):=C^{\infty}_{c,\mathrm{rad	}}(\R^5).
\end{align}
The following result for the operator $L_0$ is well-known, and we refer the reader to \cite{GloSch20}, Lemma $3.1$, for a detailed proof. 

\begin{lemma}
The operator $L_0: \mc D(L_0) \subset \mc H_0 \rightarrow \mc H_0$ is closable, and its closure $(\mc L_0,\mc D(\mc L_0))$ generates a strongly continuous  semigroup $(S_0(\tau))_{\tau \geq 0}$ of bounded operators on $\mc H_0$. Explicitly,
	\begin{equation*}
		[S_0(\tau)f](x)=e^{-\tau}(G_\tau\ast f)(e^{-\tau/2}x),
	\end{equation*}
where $G_\tau(x)=[4\pi \alpha(\tau)]^{-\frac{5}{2}}e^{-|x|^2/4\alpha(\tau)}$ and $\alpha(\tau)=1-e^{-\tau}$.
\end{lemma}
 The operator $L$, on the other hand, has a self-adjoint realization in the space $\mc H$; we have the following fundamental result.

\begin{proposition}\label{Prop:Spectral_Prop}
		The operator $ L: \mc D(L) \subset \mc H \rightarrow \mc H$ is closable, and its closure $(\mc L,\mc D(\mc L))$ generates a  strongly continuous  semigroup $(S(\tau))_{\tau \geq 0}$ of bounded operators on $\mc H$. The spectrum of $\mc L$ consists of a discrete set of eigenvalues, and moreover
\begin{equation}\label{Eq:Spec_L}
\sigma(\mc L) \subset (-\infty,0) \cup \{1\},
\end{equation}
	where $\la=1$ is a simple eigenvalue with the normalized eigenfunction $g$ from \eqref{Def:P}.    Furthermore, for the orthogonal projection $\mc P : \mc H \rightarrow \mc H$ defined in \eqref{Def:P} there exists $\omega_0 >0$  such that
	\begin{equation}\label{Eq:S_decay_H}
		\|	S(\tau) (1- \mc P)  f \|_{\mc H }  \leq e^{-\omega_0 \tau} \|(1- \mc P)  f \|_{\mc H} \quad \text{and} \quad S(\tau) \mc P  f =  e^{\tau} \mc Pf,
	\end{equation}
for all $f \in \mc H$ and all  $\tau \geq 0$.
	\end{proposition}
	
\begin{proof}
We define the unitary map
\[
U: L^2(\R^+) \to \mc H, \quad u \mapsto Uu = |\mathbb S^{4}|^{-\frac12} |\cdot|^{-2} e^{\frac{|\cdot|^2}{8}} \phi(|\cdot|) u(|\cdot|) 
\]
and note that $-L = U A U^{-1}$ with
\begin{align*}
A u(r) = - u''(r) + q(r) u(r)
\end{align*}
where 
\[ q(r) := \frac{2}{r^2} + \frac{r^2}{16} - \frac54 -\frac{16}{(2+r^2)^2}-\frac{4}{2+r^2}, \]
and $\mc D(A) = U^{-1} \mc D(L)$. Let us denote by $A_c$ the restriction of $A$ to $C^{\infty}_c(\R^+)$. Using standard results, we infer that $A_c$ is limit-point at both endpoints of the interval $(0,\infty)$ (see, e.g., \cite{Wei87}, Theorem 6.6, p.~96, and Theorem 6.4, p.~91). Hence, the unique self-adjoint extension of $A_c$ is given by its closure, which is the maximal operator $\mc A: \mc D(\mc A) \subset L^2(\R^+) \to L^2(\R^+)$, where
\[ \mc D(\mc A) := \{ u \in  L^2(\R^+): u,u' \in AC_{\mathrm{loc}}(\R^+), A u \in L^2(\R^+) \}, \]
and $\mc A u = A u$ for $u \in \mc D(\mc A)$. The inclusion $A_c \subset A \subset \mc A = \overline{A_c}$ implies $\overline{A} =\mc A$. Now, since $q$ is bounded from below, the same is true for the operator $\mc A$, i.e, there is a constant
 $\mu > 0$ such that $\mathrm{Re}\langle \mc A u, u \rangle_{L^2(\R^+)} \geq  -\mu \| u \|_{L^2(\R^+)}$ for all $u \in \mc D(\mc A)$. Since $q(r) \rightarrow +\infty$ when $r \to \infty$, $\mc A$ has compact resolvent, and its spectrum therefore consists of a discrete set of eigenvalues. The analogous properties of $L$ follow by the unitary equivalence. In particular, $ L: \mc D(L) \subset \mc H \rightarrow \mc H$ is essentially self-adjoint and its closure is given by $\mc L = - U \mc A U^{-1}$ with $\mc D(\mc L) =  U \mc D(\mc A)$. Moreover, we have that $\mathrm{Re}\langle \mc L f, f \rangle_{\mc H} \leq \mu \| f \|_{\mc H}$ for all $f \in \mc D(\mc L)$. This implies that $\mc L$, being self-adjoint, generates a strongly continuous semigroup  on $\mc H$. 

Next, we describe the spectral properties of $\mc L$. Obviously, $g \in \mc D(\mc L)$ and $ L g = g$ by explicit calculation. Hence,  $\tilde {g} := U^{-1} g $ satisfies $(1 + A) \tilde {g} = 0$. Moreover, $\tilde {g}$ is strictly positive on $(0, \infty)$. Hence, we have the factorization $A = A^{-} A^{+}-1$, where 
\[ A^{\pm} = \pm \partial_r  - \frac{ \tilde {g}'(r)}{\tilde { g}(r)}. \]
We define $ A_{\mc S} := A^{+} A^{-} - 1$. Explicitly, we have 
\begin{align*}
A_{\mc S} u(r) = - u''(r) + \frac{6}{r^2} u(r) + Q_S(r) u(r),
\end{align*} 
with
\[Q_S(r):= \frac{r^2}{16}-\frac{3}{4} - \frac{8}{2+r^2}.\]
This gives rise to the maximally defined self-adjoint operator $\mc A_{\mc S}: \mc D(\mc A_{\mc S}) \subset L^2(\R^+) \to L^2(\R^+)$ which is, by construction, isospectral to $\mc A$ except for $\lambda = -1$. For a detailed discussion on the above process of ``removing" an eigenvalue, see \cite{GloSch21}, Section B.1. Now, to prove \eqref{Eq:Spec_L}, it is enough to show that $\mc A_S$ does not have any eigenvalues in $(-\infty,0]$.
For this, we use an adaptation of the so-called GGMT criterion; see \cite{GloSch20}, Theorem A.1.  In particular, we show that for $p=2$, 
\begin{equation}\label{Eq:Assum}
\int_0^{\infty} r^{2p-1} |Q_S^-(r)|^p dr <\frac{(4\alpha+1)^{\frac{2p-1}{2}}p^p \Gamma(p)^2}{(p-1)^{p-1}\Gamma(2p)}
\end{equation} 
where $\alpha = 6$ and $Q_S^-(r) = \min \{ Q_S(r), 0 \}$. By calculating explicitly the root of $Q_S$ one can easily check that $Q_S(r) > 0$ for $r \geq 5$, hence
\[\int_0^{\infty} r^{3} Q_S^-(r)^2 dr < \int_0^{5}  r^{3} Q_S(r)^2 dr.\]
By noting that 
\[ r^3 Q_S(r)^2 = \frac{r^7}{256}-\frac{r^5}{r^2+2}-\frac{3 r^5}{32} +\frac{9 r^3}{16} +\frac{12 r^3}{r^2+2}+\frac{64
   r^3}{\left(r^2+2\right)^2}, \]
we can calculate the integral explicitly and find that 
\[\int_0^{5}  r^{3} Q_S(r)^2 dr = \frac{1305275}{55296}-18 \ln (2)+18 \ln (27) < \frac{250}{3} \]
where the right hand side of this inequality is the  value of the right hand side of Eq.~\eqref{Eq:Assum} for $p=2$ and $\alpha=6$. Theorem A.1 in \cite{GloSch20} now  implies that $\sigma(\mc A_S) \subset (0,+\infty)$, and  \eqref{Eq:Spec_L} follows. Furthermore, a simple ODE analysis yields that the geometric eigenspace of $\la=1$ is equal to $\langle g \rangle$. Consequently, $\mc P$ is the orthogonal projection onto $\rg \mc P$ in $\mc H$, and we readily get \eqref{Eq:S_decay_H}.
\end{proof}

Next statement shows that the growth bounds are preserved when  the linear evolution is measured in graph norms associated to fractional powers of the operator $1- \mc L$. This result will be crucial in Section \ref{Evol:Xk}, in particular, in conjunction with Lemma \ref{Le:Graphnorm_Xk} below. 
 
\begin{lemma}\label{Le:Graph_L}
There is a unique self-adjoint, positive operator $(1 - \mc L)^{\frac{1}{2}}$ with  $C^\infty_{c,\mathrm{rad}}(\R^5)$ as a core, such that $\big((1 - \mc L)^{\frac{1}{2}}\big)^2 = 1 - \mc L$.  For $k \in \N_0$ and $f \in \mc D((1- \mc L)^{k/2})$ we define the graph norm
\begin{align*}
\| f\|_{\mc G((1- \mc L)^{k/2})} := \|f \|_{\mc H} + \|(1- \mc L)^{k/2} f \|_{\mc H},
\end{align*}
and infer that
\[\|S(\tau)  (1 -\mc P)  f \|_{\mc G((1- \mc L)^{k/2})}  \leq  e^{-\omega_0 \tau} \|(1 - \mc P) f \|_{\mc G((1- \mc L)^{k/2})} \]
for every $k \in \N_0$, $f \in \mc D((1- \mc L)^{k/2})$ and all  $\tau \geq 0$.
\end{lemma}

\begin{proof}
By standard results (see, e.g., \cite{kato}, p.~281, Theorem 3.35), 
the square root  $(1 - \mc L)^{\frac{1}{2}}$  exists and commutes with any bounded operator that commutes with $\mc L$. 
We show that $C^\infty_{c,\mathrm{rad}}(\R^5)$ is a core of $(1-\mc L)^\frac{1}{2}$.
Let $\varepsilon > 0$. Since $\mc D(\mc L)$ is core for $(1 - \mc L)^{\frac{1}{2}}$ and $C^\infty_{c,\mathrm{rad}}(\R^5)$ is a core for $1 - \mc L$ there is $\tilde f \in C^\infty_{c,\mathrm{rad}}(\R^5)$ such that 
 $\|f - \tilde f \|_{\mc H} + \|(1- \mc L)^{\frac{1}{2}} (f - \tilde f) \|_{\mc H} < \varepsilon$, by using that $\| (1- \mc L)^{\frac{1}{2}} f\|_{\mc H} \lesssim \| (1- \mc L)f\|_{\mc H}  + \|f \|_{\mc H}$. The growth bounds for the  semigroup can be proved by induction using the fact that $(1-\mc L)^{\frac12}$ commutes with the projection and the semigroup. 
\end{proof}

We conclude this section by proving two technical results that will be crucial in the sequel.

\begin{lemma}\label{Le:BoundedDom_Norm}
Let $k \in \N_0$ and $R >0$. Then 
\[
 \| \partial^{\alpha} f \|_{L^2(\B^5_R)} \lesssim \sum_{j=0}^{k} \| f \|_{\mc G((1-\mc L)^{j/2})},
 \]
 for all $f \in C^{\infty}_{c,\mathrm{rad}}(\R^5)$ and all $\alpha \in \N_0^5$ with $|\alpha| = k$.
\end{lemma}

\begin{proof}
First, for $f \in C^{\infty}_{c,\mathrm{rad}}(\R^5)$  and $\mu(x) := |x|^4 e^{-|x|^2/4 }$ we define
\begin{align*}
Bf(x) :=  \phi(x)^{2} \frac{d}{d|x|} \left ( \phi(x)^{-2} f(x) \right ), \quad B^*f(x) := - \mu(x)^{-1} \frac{d}{d|x|} \left( \mu(x)   f(x) \right).
\end{align*} 
By inspection, it follows that $1 - L = B^* B$ and thus,  
\begin{align*}
\|(1-\mc L)^{\frac12} f \|^2_{\mc H} = \langle (1-\mc L) f,f \rangle_{\mc H}  = \langle B^*B f, f \rangle_{\mc H} = \langle B f, B f \rangle_{\mc H} = \|B f\|^2_{\mc H},
\end{align*}
where the fact that $B$ and $B^*$ are formally adjoint follows from a straightforward calculation. Hence, 
\begin{align*}
\|(1-\mc L)^{\frac12} f \|_{\mc H} = \|B f\|_{\mc H}
\end{align*}
for $f \in C^{\infty}_{c,\mathrm{rad}}(\R^5)$. With this at hand, we prove the lemma by induction. For $k =0$ the inequality is immediate, and for $k=1$ we have that 
\begin{align*}
\| \partial_i f \|_{L^2(\B^5_R)}  \lesssim \|  \tilde f'(|\cdot|) \|_{L^2(0,R)} \lesssim  \|Bf\|_{\mc H} + \|f \|_{\mc H} = \|(1-\mc L)^{\frac12} f \|_{\mc H} + \|f \|_{\mc H}.
\end{align*}
Assume that the claim holds up to some  $k \geq 1$. Then we have that for all $1 \leq j \leq k$
\begin{align*}
\| D^{j+1} f\|_{L^2(\B^5_R)}  = & \| D^{j-1}D^2f \|_{L^2(\B^5_R)} \lesssim \| D^{j-1}(1-\mc L)f \|_{L^2(\B^5_R)} + \| D^{j-1}f \|_{L^2(\B^5_R)} \\ &+
\| D^{j-1}\Lambda f \|_{L^2(\B^5_R)} + \| D^{j-1}\Lambda (\phi f) \|_{L^2(\B^5_R)} + \| D^{j-1}(\phi f) \|_{L^2(\B^5_R)} \\
& \lesssim  \sum_{|\beta|\leq j-1}\| \partial^{\beta} (1- \mc L)f \|_{L^2(\B^5_R)} + \sum_{|\beta|\leq j} \| \partial^\beta f \|_{L^2(\B^5_R)} \\
& \lesssim \sum_{i=0}^{j-1}\| (1- \mc L)f \|_{\mc G((1-\mc L)^{i/2})} + \sum_{i=0}^{j}\| f \|_{\mc G((1-\mc L)^{i/2})} \\
& \lesssim \sum_{i=0}^{j+1}\| f \|_{\mc G((1-\mc L)^{i/2})}.
\end{align*}
Lemma \ref{Lem:Local_Sobolev} then implies the claim for $|\alpha|=k+1$.
\end{proof}

Finally, we show that graph norms can be controlled by $X^k$-norms.

\begin{lemma}\label{Le:Graphnorm_Xk}
Let $k \in \N$, $k \geq 3$. Then
\begin{align*}
\|(1-\mc L)^{\kappa/2} f \|_{\mc H} \lesssim \|f \|_{X^k}  
 	\end{align*}
for all $f \in X^k$ and all $\kappa \in \{0, \dots, k \}$.
\end{lemma}

\begin{proof}
We prove the statement for $f \in C^{\infty}_{c,\mathrm{rad}}(\R^5)$. The claim then follows by density and the closedness of $(1-\mc L)^{\kappa/2}$. First, it is easy to see that
\begin{align*}
\| (1 -\mc L)^{\kappa/2} f \|_{\mc H} \lesssim  \sum_{|\alpha|\leq \kappa} \|p_\alpha \partial^\alpha f \|_{\mc H}
\end{align*}
for smooth, radial and polynomially bounded functions $p_\alpha$. Using the  exponential decay of the weight function $\sigma$ along with interpolation, see Lemma \ref{Le:Interpolation}, and Hardy's inequality, we get 
\begin{align*}
\| (1 -\mc L)^{\kappa/2} f \|_{\mc H}  & \lesssim 
\sum_{1\leq |\alpha| \leq \kappa} \| \partial^\alpha f \|_{L^2(\R^5)} + \||\cdot|^{-1} f \|_{L^2(\R^5)}    \lesssim \|f \|_{X^k}. 
\end{align*}
\end{proof}

\subsection{The linearized evolution on $X^k$}\label{Evol:Xk}
  With the technical results from above, we now show that the linear evolution of \eqref{Eq:Perturbed} is well-posed in $X^k$. Since $C^{\infty}_{c,\mathrm{rad}}(\R^5)$ is dense in $X^{k}$, we consider $L$ as defined in \eqref{Def:Domain_L_L0}.

\begin{proposition}\label{Prop:Semigroup_Sk}
Let $k \in \N$, $k \geq 3$. The operator $L: \mc D(L) \subset X^k \to X^k$ is closable and with $(\mc L_k, \mc D(\mc L_k))$ denoting the closure we have that $\mc C \subset  \mc D(\mc L_k)$.  Furthermore, the operator $\mc L_k$  generates a strongly continuous semigroup $(S_{k}(\tau))_{\tau \geq 0}$ of  bounded operators on $X^k$, which coincides with the restriction of $S(\tau)$ to $X^k$, i.e,
\[S_{k}(\tau)= S(\tau)|_{X^{k}}\]
for all $\tau \geq 0$. 
\end{proposition}

\begin{proof}
We prove the first part of the statement by an application of the Lumer-Phillips theorem. For this, we show that 
\begin{equation}\label{Inequ:DissEst0}
\Re \langle L f,f\rangle_{X^k} \leq \bar{\omega}_k \| f \|^2_{X^k}
\end{equation}
for some $\bar{\omega}_k >0$ and all $f \in \mc D(L)$. In fact, we prove  a more general estimate, which will be instrumental in proving Theorem \ref{Thm:DecayX} later on.  More precisely, we show that for $R \geq 1$, there are constants $C_k, C_{R,k} > 0$ such that for all $f \in C^{\infty}_{c,\mathrm{rad}}(\R^5)$,
\begin{align}\label{Inequ:DissEst}
\Re \langle L f, f \rangle _{X^k} \leq (- \tfrac{1}{4} + \tfrac{C_k}{R^2} )\|f \|^2_{X^k} + C_{R,k} \sum_{j=0}^{k} \|f \|^2_{\mc G((1- \mc L)^{j/2})}.
\end{align}
An application of Lemma \ref{Le:Graphnorm_Xk} to Eq.~\eqref{Inequ:DissEst} immediately implies Eq.~\eqref{Inequ:DissEst0}. 
Eq.~\eqref{Inequ:DissEst} will be proved in several steps. First, 
recall that $L = L_0  + L'$. By partial integration,
	\begin{align}\label{Eq:Laplacian}
	\begin{split}
		\langle D^k \Delta f, D^k f \rangle_{L^2(\R^5)} & =\langle D^{k+2} f, D^k f\rangle_{L^2(\R^5)} = -\|D^{k+1} f\|^2_{L^2(\R^5)}.
\end{split}
	\end{align}
Based on the identity
\begin{equation*}
	\Re \langle \Lambda f,f \rangle_{L^2(\R^5)} =  -\tfrac{5}{2} \| f\|^2_{L^2(\R^5)}, 
\end{equation*}
which easily follows from integration by parts and the commutator relation Eq.~\eqref{Comm_DLambda}, we get
\begin{equation}\label{Eq:Lambda_identity}
	\Re \langle D^{k} \Lambda f,D^{k} f \rangle_{L^2(\R^5)} = (k -\tfrac{5}{2}) \|D^{k} f \|^2_{L^2(\R^5)}.
\end{equation}
Consequently, 
\begin{align}\label{Eq:Diss_Est_L0}
\Re \langle L_0  f, f \rangle _{X^k} \leq  - \tfrac{1}{4} \|f \|^2_{X^k} 
\end{align}
and thus 
\begin{align*}
\Re \langle L f, f \rangle _{X^k}  = \Re \langle (L_0 + L') f, f \rangle _{X^k} \leq  - \tfrac{1}{4} \|f \|^2_{X^k}  + \Re \langle L' f, f \rangle _{X^k}.
\end{align*}
To obtain  Eq.~\eqref{Inequ:DissEst0}, one can estimate the part containing $L'$ in $X^k$ in a straightforward manner. However, for the refined bound \eqref{Inequ:DissEst}, the argument is more involved.
 We write
\begin{align}\label{Eq:Diss_Est_Pert}
\Re \langle L' f, f \rangle _{X^k} = \Re \langle V f, f \rangle _{X^k} + 2 \Re \langle \phi \Lambda f, f \rangle _{X^k}.
\end{align} with
\[ V(x) := 2 x \cdot \nabla \phi(x) + 12 \phi(x). \]
 Note that $V \in C^{\infty}_{\mathrm{rad}}(\R^5)$ and the properties of $\phi$ imply that 
\begin{align}\label{Est:Decay_V}
|\partial^{\alpha} V(x)| \lesssim \langle x \rangle^{-2-|\alpha|},
\end{align} 
for $\alpha \in \N_0^5$. First, we prove that 
\begin{align}\label{Est1}
\Re \langle D^k (V f), D^k f \rangle _{L^2(\R^5)} \leq C_{R} \sum_{j=0}^{k} \| f \|^2_{\mc G((1-\mc L)^{j/2})} +
\tfrac{C}{R^2}  \|f \|^2_{X^k}
\end{align}
for suitable constants $C_{R} , C  >0$. 
For this, we use the fact that
\begin{align*}
\Re \langle D^k (V f), D^k f  \rangle _{L^2(\R^5)}  \lesssim \sum_{|\tilde{\alpha}|=|\alpha| = k }  | \Re \langle \partial^{\tilde{\alpha}} (V f),  \partial^{\alpha} f \rangle _{L^2(\R^5)}|,
\end{align*}
for all $f \in C^{\infty}_{c,\mathrm{rad}}(\R^5)$. 
Hence, we can apply the standard Leibniz rule to estimate products. More precisely, for $\alpha,\beta,\gamma \in \N_0^5$ for which $|\alpha|=|\beta+\gamma| = k$, we estimate
\begin{align*}
|\langle \partial^{\beta} V \partial^{\gamma} f, \partial^{\alpha}  f \rangle_{L^2(\R^5)}| \leq  \||\partial^{\beta} V|^{\frac{|\beta|+1}{|\beta|+2}} \partial^{\gamma} f\|^2_{L^2(\R^5)} + 
\||\partial^{\beta} V|^{\frac{1}{|\beta|+2}} \partial^{\alpha} f \|^2_{L^2(\R^5)}.
\end{align*}
By Eq.~\eqref{Est:Decay_V} and Lemmas \ref{Le:BoundedDom_Norm} and \ref{Le:Interpolation} the last term can be bounded by
\begin{align*}
\||\partial^{\beta}   V|^{\frac{1}{|\beta|+2}} \partial^{\alpha} f \|^2_{L^2(\R^5)}   &\leq C   \| \partial^{\alpha} f \|^2_{L^2(\B^5_R)} + \tfrac{C}{R^2}  \| \partial^{\alpha} f \|^2_{L^2(\R^5 \setminus \B^5_R)} \\
&  \leq C_{R} \sum_{j=0}^{k} \| f \|^2_{\mc G((1-\mc L)^{j/2})} +
\tfrac{C}{R^2}  \|f \|^2_{X^k}.
\end{align*}
Similarly,  we have that
\begin{align*}
\||\partial^{\beta} V|^{\frac{|\beta|+1}{|\beta|+2}} \partial^{\gamma} f\|^2_{L^2(\R^5)} \leq C_{R} \sum_{j=0}^{|\gamma|} \| f \|_{\mc G((1-\mc L)^{j/2})} + \tfrac{C}{R^2}   \||\cdot|^{-|\beta|} \partial^{\gamma} f \|^2_{L^2(\R^5 \setminus \B^5_R)}.
\end{align*}
For $\gamma = 0$, we estimate the last term by Hardy's inequality,
\[  \||\cdot|^{-k}  f \|_{L^2(\R^5 \setminus \B^5_R)} \leq 
 \||\cdot|^{-1}  f \|_{L^2(\R^5 \setminus \B^5_R) } \leq  \||\cdot|^{-1}  f \|_{L^2(\R^5) }  \lesssim  \| D f\|_{L^2(\R^5)} \lesssim  \|f \|_{X^k}. \]
For $\gamma \neq 0$, by interpolation and equivalence of norms, see Lemma ~\eqref{Le:Interpolation}, we obtain
\[ 
\||\cdot|^{-|\beta|} \partial^{\gamma} f \|_{L^2(\R^5 \setminus \B^5_R)} \leq \| \partial^{\gamma} f \|_{L^2(\R^5 \setminus \B^5_R)} \lesssim    \|f \|_{X^k}, 
\]
which implies Eq.~\eqref{Est1}.
To estimate the second term in Eq.~\eqref{Eq:Diss_Est_Pert}, we use the following relation
\begin{align}\label{Eq:Re_est}
	\Re \langle D^k (\phi \Lambda f), D^k f  \rangle _{L^2(\R^5)}  \lesssim  | \Re \langle  \phi \Lambda D^kf, D^k f \rangle _{L^2(\R^5)}|  
	 + \sum_{\substack{1 \leq |\beta|\leq k \\ |\alpha| = k}}   | \langle \varphi_\beta \partial^\beta f,  \partial^{\alpha} f \rangle _{L^2(\R^5)}|
\end{align}
with certain smooth functions $|\varphi_{\beta}(x)| \lesssim \langle  x \rangle^{-2}$.
Note that by partial integration we have the following identity 
\begin{equation*}
\Re \, \langle \phi \Lambda D^kf,D^kf \rangle_{L^2(\R^5)} = -\tfrac{5}{2}\langle \phi D^kf,D^kf\rangle_{L^2(\R^5)} - \tfrac{1}{2}\langle (\Lambda\phi) D^kf,D^kf \rangle_{L^2(\R^5)}.
\end{equation*}
Therefore, the first term in \eqref{Eq:Re_est} can be estimated
\begin{align*}
| \Re \langle  \phi \Lambda D^k f,   D^k f \rangle _{L^2(\R^5)}| \lesssim \| \langle \cdot \rangle^{-1} D^k f \|^2_{L^2(\R^5)}.
\end{align*}
For the second term we have
\begin{align*}
 |  \langle  \varphi_{\beta}  \partial^{\beta} f, \partial^{\alpha}  f \rangle _{L^2(\R^5)} |    \lesssim  \| \langle \cdot \rangle^{-1}  \partial^{\beta} f \|^2_{L^2(\R^5)} +   \|\langle \cdot \rangle^{-1}   \partial^{\alpha}  f  \|^2_{L^2(\R^5)}.  
\end{align*}
Based on this, similarly to above we infer that 
\begin{align}\label{Est2}
\Re \langle D^k (\phi \Lambda f), D^k f  \rangle _{L^2(\R^5)} \leq C_R \sum_{j=0}^{k} \| f \|_{\mc G((1-\mc L)^{j/2})} +
\tfrac{C}{R^2}  \|f \|^2_{X^k},
\end{align}
for suitably chosen $C_R, C >0$. Using the same arguments, we arrive at \eqref{Est1} and \eqref{Est2} for  $D$ instead of $D^k$, and we hence get \eqref{Inequ:DissEst}, and thereby \eqref{Inequ:DissEst0} as well. 
From this, we infer that the operator $L$ is closable (see, e.g., \cite{EngNag00}, p.~82, Proposition 3.14-(iv)), and that 
the closure $\mc L_k  : \mc D(\mc L_k) \subset X^k \to X^k$, satisfies
\begin{equation}\label{Inequ:DissEst1}
\Re \langle \mc  L_k f,f\rangle_{X^k} \leq \bar{\omega}_k \| f \|^2_{X^k}
\end{equation}
for all $f \in \mc D(\mc L_k)$ and some suitable $\bar{\omega}_k > 0$.

\smallskip

Next, we prove that $\mc C \subset \mc D(\mc L_k)$ by showing that for  $f \in \mc C$ there is a sequence $(f_n) \subset  C^{\infty}_{c,\mathrm{rad}}(\R^5)$, such that $f_n \to f$  and $L f_n \to v$ in $X^k$. Then, by definition of the closure, $f \in \mc D(\mc L_k)$ and $\mc L_k f = v$. We set $f_n =  \chi(\cdot/n) f$, where  $\chi$ a  smooth, radial cut-off function equal to one for $|x| \leq 1$ and zero for $|x| \geq 2$. It is easy to see that $(f_n) \subset C^{\infty}_{c,\mathrm{rad}}(\R^5)$ converges to $f$ in  $X^k$. More precisely, by exploiting the decay of $f$,
\begin{align*}
\| f_n - f \|_{L^{\infty}(\R^5)} =  \| f_n - f \|_{L^{\infty}(\R^5\setminus \B^5_n)} \lesssim  \|f \|_{L^{\infty}(\R^5\setminus \B^5_n)} \lesssim n^{-2},
\end{align*}
we have $f_n \to f$ in $L^{\infty}(\R^5)$. Furthermore,  for $m \leq n$, $f_n - f_m = 0$ on $\B^5_m$ and thus 
\[ \| D(\chi_n f)  - D(\chi_m f)\|_{L^2(\R^5)} = \| D(\chi_n f)  - D(\chi_m f)\|_{L^2(\R^5 \setminus \B^5_m)}. \]
We have 
\begin{align*}
\| D(\chi_n f) \|_{L^2(\R^5\setminus \B^5_m)} \lesssim \||\cdot|^{-2} D\chi_n \|_{L^2(\R^5\setminus \B^5_m)} + \| |\cdot|^{-3} \chi_n \|_{L^2(\R^5\setminus \B^5_m)} \lesssim n^{-\frac12},
\end{align*}
and similarly  $\| D^{k}(\chi_n f) \|_{L^2(\R^5\setminus \B^5_m)} \lesssim n^{-\frac12}$. Thus,
 $\| f_n - f_m \|_{X^k} \lesssim  n^{-\frac12} +  m^{-\frac12}$
and $(f_n)$ converges in $X^k$ to some limiting function which must be equal to $f$ by the $L^{\infty}-$embedding and the uniqueness of limits.

Now, $\Lambda f \in \mc C$ for $f \in \mc C$. Since $\Lambda f_n = f \Lambda\chi_n +  \chi_n \Lambda f$ and 
\begin{align*}
\| \Lambda f_n -  \Lambda f \|_{L^{\infty}(\R^5)} \lesssim  \|\chi_n \Lambda f - \Lambda f \|_{L^{\infty}(\R^5\setminus \B_n^5)} + \|f \Lambda \chi_n \| _{L^{\infty}(\R^5\setminus \B_n^5)} \lesssim n^{-2},
\end{align*}
we have $\Lambda f_n \in  C^{\infty}_{c,\mathrm{rad	}}(\R^5) \to \Lambda f$ in $L^{\infty}(\R^5)$. By similar considerations as above one finds that $(\Lambda f_n)$ is Cauchy in $X^k$ and thus $\Lambda f_n \to \Lambda f$ in $X^k$ for $n \to \infty$. 
Now, 
\begin{align}\label{Eq:Op_Seq}
 L f_n = \Delta f_n  - \frac12 \Lambda f_n - f_n + 2 \Lambda  (\phi f_n) + 12 \phi f_n.
\end{align}
Arguments as above and the Banach algebra property of $X^k$ imply that $(L f_n)$ converges in $X^k$ to some limiting function $v \in X^k$. By Sobolev embedding, $L f_n \to v$ in $L^{\infty}(\R^5)$ and by convergence of the individual terms,
$v = \Delta f  - \frac12 \Lambda f - f + 2 \Lambda  (\phi f) + 12 \phi f.$ This shows in particular, that $\mc L_k$ acts as a classical differential operator on $\mc C$. 

\smallskip

For the invocation of the Lumer-Phillips theorem, it is left to prove the density of the range of $\lambda_k - \mc L_k$ for some $\lambda_k > \bar{\omega}_k$. This crucial property is established by an ODE argument, the proof of which is rather technical and therefore provided in Appendix \ref{Sec:App_ODE}. More precisely, let $f \in C^{\infty}_{c,\mathrm{rad	}}(\R^5)$ such that $f = \tilde f (|\cdot|)$. By Lemma \ref{Ap:ODE}, there exists $\la > \bar{\omega}_k$ such that the ODE
\begin{align*}
 \tilde  u''(\rho) + \left(\frac{4}{\rho} - \frac{1}{2} \rho  + 2 \rho \tilde  \phi(\rho) \right ) \tilde  u'(\rho) +  \left(  2\rho \tilde  \phi'(\rho) + 12 \tilde  \phi(\rho) - (\lambda_k + 1) \right) \tilde  u(\rho) =  - \tilde  f(\rho),
\end{align*}
with $\phi = \tilde \phi(|\cdot|)$ has a solution $\tilde u \in C^1[0,\infty) \cap C^{\infty}(0,\infty)$ satisfying $\tilde u'(0) = 0$ as well as  $\tilde u^{(j)}(\rho) = \mc O(\rho^{-3-j})$ for $j \in \N_0$ as $\rho \to \infty$. By setting $u := \tilde u(|\cdot|)$, we obtain a classical solution to the equation 
\begin{equation}\label{Eq:Elliptic}
	(\lambda_k - L) u = f
\end{equation} on $\R^5\setminus \{0\}$. Since $u$ belongs to $H^1(\R^5)$, it solves \eqref{Eq:Elliptic} weakly on $\R^5$, and by elliptic regularity we infer that $u \in C^{\infty}_{\mathrm{rad}}(\R^5)$. The decay of $\tilde u$ at infinity implies that $u \in \mc C$. Hence, $u \in \mc D( \mc L_k)$ which implies the claim. 

An application of the Lumer-Phillips Theorem  now proves that $(\mc L_k, \mc D(\mc L_k))$ generates a  strongly continuous semigroup $(S_k(\tau))_{\tau \geq 0}$ on $X^k$. In view of the embedding $X^k \hookrightarrow \mc H$  and the fact that $C^{\infty}_{c,\mathrm{rad}}(\R^5)$ is a core for $\mc L$ and $\mc L_{k}$ for any $k$, an application of Lemma C.1 in \cite{Glo22a} proves the claimed restriction properties. 
\end{proof}

In view of the restriction properties stated in Proposition \ref{Prop:Semigroup_Sk}, we can safely omit the index $k$ in the notation of the semigroup.

Before turning to the proof of Theorem \ref{Thm:DecayX}, we state a result for the free evolution, which follows in a straightforward manner analogous to the proof of Proposition \ref{Prop:Semigroup_Sk}, using in particular Eq.~\eqref{Eq:Diss_Est_L0} and setting $L'=0$ in the subsequent arguments.

\begin{lemma}\label{Lem:S_0_decay}
Let $k \in \N$, $k \geq 3$. The operator $L_0: \mc D(L_0) \subset X^k \to X^k$ is closable and its closure $(\mc L_{0,k}, \mc D(\mc L_{0,k}))$ generates a strongly continuous semigroup on $X^k$, which coincides with the restriction of the $S_0(\tau)$ to $X^k$ for $\tau \geq 0$. Furthermore, 
\[ \|S_0(\tau) f \|_{X^k} \leq e^{-\frac{\tau}{4}} \| f \|_{X^k} \]
for all $f \in X^k$ and $\tau \geq 0$.
\end{lemma}

\subsection{Proof of Theorem \ref{Thm:DecayX}}

First, we show that the operator $\mc P$ from \eqref{Def:P} induces a non-orthogonal rank-one projection on $X^k$. To indicate the dependence on $k$, we write
\begin{equation}\label{Def:P_k}
	\mc P_{X^{k}} f := \langle f, g \rangle_{\mc H} ~ g
\end{equation}
for $f \in X^k$. By its decay properties, the function $g$ is an element of $\mc C  \subset X^k$ for any $k \in \N$. In view of the embedding $X^k \hookrightarrow \mc H$, the inner product makes sense for $f \in X^k$ and by definition $\mc P_{X^{k}}^2 = \mc P_{X^{k}}$. The fact that the projection commutes with  $S_{k}(\tau)$ follows from the respective properties on $\mc H$. 

Let $f \in C^\infty_{c,\mathrm{rad}}(\R^5)$. By Proposition \ref{Prop:Semigroup_Sk}, $\tilde{f}:=(1-\mc P_{X^{k}})f \in \mc C \subset \mc D(\mc L_k)$. Using Eq.~\eqref{Inequ:DissEst}, Lemma \ref{Le:Graph_L} and Lemma \ref{Le:Graphnorm_Xk}, we infer that for $R \geq 1$ sufficiently large,
\begin{align*}
\tfrac{1}{2}  \tfrac{d}{d\tau} \|S(\tau)\tilde{f}  \|^2_{X^k}  &   = ( \partial_{\tau} S(\tau) \tilde{f}|S(\tau) \tilde{f})_{X^k}   = ( \mc L_k S(\tau)  \tilde{f}|S(\tau) \tilde{f})_{X^k}  \\
& \leq (-\tfrac{1}{4}  + \tfrac{C_k}{R^2}) \|S(\tau) \tilde{f}\|^2_{X^k}    + 
C_{R,k} \sum_{j=0}^{k} \|S(\tau)\tilde{f} \|^2_{\mc G((1- \mc L)^{j/2})}  \\
& \leq  -\tfrac{1}{8} \|S(\tau) \tilde{f}\|^2_{X^k}+ C_{k} e^{-2 \omega_0 \tau}  \sum_{j=0}^{k} \| \tilde{f} \|^2_{\mc G((1- \mc L)^{j/2})} \\
&  \leq -\tfrac{1}{8} \|S(\tau) \tilde{f}\|^2_{X^k} +  C_{k} e^{-2 \omega_0 \tau}  \|\tilde{f} \|_{X^k}^2 \leq - 2 c  \|S(\tau) \tilde{f}\|^2_{X^k}  + C_k e^{-4 c \tau}  \|\tilde{f} \|^2_{X^k},
\end{align*}
for $c = \frac{1}{2} \min \{\omega_0, \frac{1}{8} \}$. Hence, 
\begin{align*}
\tfrac{1}{2} \tfrac{d}{d\tau}  \left[ e^{4 c \tau}   \|S(\tau) \tilde{f}\|^2_{X^k}  \right] \leq C_k \|\tilde{f} \|^2_{X^k}
\end{align*}
and integration yields
\[   \|S(\tau) \tilde{f}\|^2_{X^k} \leq (1 + 2 C_k \tau ) e^{-4 c  \tau}  \|\tilde{f} \|^2_{X^k}  \lesssim e^{-2 \omega_k \tau}   \|\tilde{f} \|^2_{X^k} ,\]
for some suitably chosen $\omega_k  > 0$ and all $\tau \geq 0$. For $f \in X^k$, the same bound follows by density. Again, for simplicity, we write $\mc P f= \mc P_{X^k} f$ for $f \in X^k$.


\section{Nonlinear estimates}\label{Sec:Nonlin}
\noindent Now we turn to the analysis of the full nonlinear equation \eqref{Eq:Perturbed}. In this section, we establish for the operator $N$ a series of estimates which will be necessary later on for constructing solutions to \eqref{Eq:Perturbed}.
Recall that for $k \geq 3$ the space $X^k$ embeds into $C^{k-3}_{\text{rad}}(\R^5)$. Therefore, multiplication and taking derivatives of order at most $k-3$ is well defined for functions in $X^k$. With this in mind, we formulate and prove the following important lemma.

\begin{lemma}\label{Lem:Lambda(fg)_est}
	Let $k \in \N, k \geq 4$. Given $f,g \in X^k$ we have that $\Lambda(fg) \in X^{k-1}$. Furthermore,
	\begin{equation}\label{Eq:Lambda_est}
			\| \Lambda(fg) \|_{X^{k-1}} \lesssim \| f\|_{X^{k}} \| g \|_{X^{k}},
\end{equation} 
for all $f,g \in X^k$.
\end{lemma}
\begin{proof}
  In this proof we crucially rely on a recently established inequality for the weighted $L^\infty$-norms of derivatives of radial Sobolev functions; see \cite{Glo22a}, Proposition B.1. For convenience, we copy here the version of this result in five dimensions. Namely, given $s \in (\frac12,\frac52)$ and $\alpha \in \N_0^5$, we have that
 \begin{equation}\label{Eq:Strauss}
 	\| |\cdot|^{\frac{5}{2}-s} \partial^\alpha u\|_{L^\infty(\R^5)} \lesssim \| u \|_{\dot{H}^{|\alpha|+s}(\R^5)} 
 \end{equation}
 for all $u \in C^\infty_{c,\text{rad}}(\R^5)$.
 Now we turn to proving \eqref{Eq:Lambda_est}. Due to the $W^{1,\infty}$-embedding of $X^k$ for $k \geq 4$ and the fact that $(f,g) \mapsto \Lambda(fg)$ is bilinear, it is enough to show \eqref{Eq:Lambda_est} for $f,g \in C^\infty_{c,\text{rad}}(\R^5)$.
 To estimate the $\dot{H}^{k-1}$ part, we do the following. If $k$ is odd, then
	\begin{align*}
		\|  D^{k-1} \Lambda (fg) \|_{L^2(\R^5)} &= \|  \Delta^{\frac{k-1}{2}} \Lambda (fg) \|_{L^2(\R^5)} \\
		&\lesssim \sum_{|\alpha| + |\beta| = k} \| |\cdot| \partial^\alpha f \, \partial^\beta g \|_{L^2(\R^5)} + \sum_{|\alpha| + |\beta| = k-1} \| \partial^\alpha f \, \partial^\beta g \|_{L^2(\R^5)}\\
		& \lesssim \| f \|_{X_k}\| g \|_{X_k},
	\end{align*}
	where the last estimate  follows from a combination of the $L^\infty$-embedding of $X_k$, Hardy's inequality, and the inequality \eqref{Eq:Strauss}. To illustrate this, we estimate the first sum above. Without loss of generality we assume that $|\alpha| \leq |\beta|$. We then separately treat the integrals corresponding to the unit ball and its complement. For the unit ball, we first assume that $\alpha=0$. Then
	\begin{align*}
		\| |\cdot| f \, \partial^\beta g \|_{L^2(\B^5)} \leq \| f \|_{L^\infty(\R^5)} \| \partial^\beta g\|_{L^2(\R^5)} \lesssim \| f \|_{X_k}\| g \|_{X_k},
	\end{align*}
by the $L^\infty$-embedding. If $\alpha \neq 0$, we have that
	\begin{align}
	\| |\cdot| \partial^\alpha f \, \partial^\beta g \|_{L^2(\B^5)} \leq \||\cdot|^\frac{3}{2} \partial^\alpha f \|_{L^\infty(\R^5)} \| |\cdot|^{-1} \partial^\beta g\|_{L^2(\R^5)} \lesssim \| f \|_{X_k}\| g \|_{X_k},
\end{align}
by \eqref{Eq:Strauss} and Hardy's inequality. For the complement of the unit ball, we have
	 \begin{align*}
	 \| |\cdot| \partial^\alpha f \partial^\beta g \|_{L^2(\R^5\setminus \B^5)} 
		& \leq 	\| |\cdot|^{\frac{3}{2}}\partial^\alpha f \|_{L^\infty(\R^5)}  \| \partial^\beta g \|_{L^2(\R^5)} \lesssim \| f \|_{X_k}\| g \|_{X_k}
	\end{align*}
	by \eqref{Eq:Strauss} only. The second sum is estimated similarly.
	 If $k$ is even, then we have that
	\begin{equation*}
		\|  D^{k-1} \Lambda (fg) \|_{L^2(\R^5)} = \| \nabla \Delta^{\frac{k-2}{2}}\Lambda (fg)   \|_{L^2(\R^5)},
	\end{equation*}
	and the desired estimate follows similarly to the previous case. The $\dot{H}^1$ part of the norm is treated in the same fashion.
\end{proof}
	Now we establish the crucial smoothing properties of $S_0(\tau)$. 
	\begin{proposition}\label{Prop:S_0_est}
		Let $k \in \N, k \geq 4$. Then, given $f,g \in X^k$ the function
	\begin{equation}\label{Def:Map_SL}
		\tau \mapsto S_0(\tau)\Lambda (fg)
	\end{equation} 
	maps $(0,\infty)$ continuously into $X^{k+1}$. Furthermore, denoting $\beta(\tau):=\alpha(\tau)^{-\frac12}$, we have that
	\begin{align}
		\|S_0(\tau)\Lambda(fg) \|_{X^{k+1}} &\lesssim  \beta(\tau) e^{-\frac{\tau}{4}}\|f \|_{X^k}\|g \|_{X^k}, \label{Eq:S_0_Lambda_est}\\
	\|S_0(\tau)\Lambda(fg) \|_{X^{k}} &\lesssim e^{-\frac{\tau}{4}}\|f \|_{X^k}\|g \|_{X^k}, \label{S_0_Lambda_estimate_1} 
	\end{align}
for all $\tau >0$ and all $f,g \in  X^k$.
\end{proposition}
\begin{proof}
	Similarly to above, by the embedding \eqref{Eq:Embed_Xk} and the underlying linearity, it is enough to show the proposition for $f,g \in C^\infty_{c,\text{rad}}(\R^5)$. First, note that for $u \in  C^\infty_{c,\text{rad}}(\R^5)$ and $v \in \mc S_{\text{rad}}(\R^5)$ the following relation holds
	\begin{equation*}
		\Lambda (u \ast v) = (\Lambda u)\ast v + u \ast (\Lambda v) + 5 \, u \ast v.
	\end{equation*}
Accordingly, we have that
\begin{align}
	S_0(\tau) \Lambda u &= e^{-\tau} [G_{\alpha(\tau)} \ast \Lambda u] (e^{-\frac{\tau}{2}}\cdot)  \nonumber \\
	&= e^{-\tau} [\Lambda (G_{\alpha(\tau)}\ast u) - (\Lambda G_{\alpha(\tau)}) \ast u - 5\, G_{\alpha(\tau)} \ast u](e^{-\frac{\tau}{2}}\cdot) \nonumber \\
	&=\Lambda S_0(\tau)u - \tilde{S}_0(\tau)u - 5 \, S_0(\tau)u, \label{Eq:S_tilde}
\end{align}
where $\tilde{S}_0(\tau)$ is given by the (scaled) convolution with the radial Schwartz function $\tilde{G}_{\alpha(\tau)}:=\Lambda G_{\alpha(\tau)}$. To prove the estimate \eqref{Eq:S_0_Lambda_est} we do the following.
First, we note that the $L^1(\R^5)$-norm of $\tilde{G}_{\alpha(\tau)}$ does not depend on $\tau$, so by Young's inequality we have that
\begin{equation*}
	\| \tilde{S}_0(\tau) u \|_{\dot{H}^1(\R^5)} \lesssim e^{-\frac{\tau}{4}}\| u \|_{\dot{H}^1(\R^5)}.
\end{equation*}
For the $\dot{H}^{k+1}$-norm, again by Young's inequality we have that
\begin{equation*}
 \| D^{k+1} \tilde{S}_0(\tau)u \|_{L^2(\R^5)} \lesssim e^{-\frac{\tau}{4}} \| D\tilde{G}_{\alpha(\tau)} \|_{L^1(\R^5)} \| D^{k} u \|_{L^2(\R^5)} \lesssim e^{-\frac{\tau}{4}} \beta(\tau) \| D^ku \|_{L^2(\R^5)}.
\end{equation*}
In summary
\begin{equation*}
		\| \tilde{S}_0(\tau) u \|_{X^{k+1}} \lesssim \beta(\tau) e^{-\frac{\tau}{4}}\| u \|_{X^k}.
\end{equation*}
In the same way we get the estimate for $S_0(\tau)u$. According to \eqref{Eq:S_tilde}, it remains to treat $\Lambda S_0(\tau)u$.
According to the commutator relation \eqref{Comm_DLambda}, to bound the $\dot{H}^1$-norm, it is enough to estimate $\Lambda D S_0(\tau)u$ in $L^2$. To that end, we have
\begin{align*}
	\| \Lambda D S_0(\tau)u \|_{L^2(\R^5)} &
	\lesssim e^{-\frac{\tau}{4}}\| |\cdot|  S_0(\tau) D(D u) \|_{L^2(\R^5)}\\
	&\lesssim e^{-\frac{\tau}{4}} \| G_{\alpha(\tau)} \|_{L^1(\R^7)} \| D(Du) \|_{L^2(\R^7)}\\
	& \lesssim e^{-\frac{\tau}{4}} \| \Lambda D u \|_{L^2(\R^5)}\\
	& \lesssim e^{-\frac{\tau}{4}}( \| D u \|_{L^2(\R^5)} + \| D\Lambda u \|_{L^2(\R^5)}),
\end{align*}
for all $\tau >0$ and $u \in C^\infty_{c,\text{rad}}(\R^5)$.
To estimate the $\dot{H}^{k+1}$-norm, it is enough to bound $\Lambda D^{k+1} S_0(\tau) u$ in $L^2$. For this, we similarly get
\begin{align*}
	\| \Lambda D^{k+1} S_0(\tau)u \|_{L^2(\R^5)} &
	\lesssim e^{-\frac{\tau}{4}}  \| D^2 G_{\alpha(\tau)} \|_{L^1(\R^7)} \|D(D^{k-1}u) \|_{L^2(\R^7)}\\
	& \lesssim e^{-\frac{\tau}{4}} \| \Lambda D^{k-1}u \|_{L^2(\R^5)}\\
	&\lesssim e^{-\frac{\tau}{4}}( \| D^{k-1} u \|_{L^2(\R^5)} + \| D^{k-1}\Lambda u \|_{L^2(\R^5)}),
\end{align*}
for all $\tau >0$ and $u \in C^\infty_{c,\text{rad}}(\R^5)$. In summary, we have that
\begin{equation*}
	\| S_0(\tau)\Lambda u \|_{X^{k+1}} \lesssim \beta(\tau) e^{-\frac{\tau}{4}} (\| u \|_{X^{k}} + \| \Lambda u \|_{X^{k-1}}).
\end{equation*}
Finally, by putting $fg$ instead of $u$, the estimate \eqref{Eq:S_0_Lambda_est} follows from the Banach algebra property of $X^k$ and Lemma \ref{Lem:Lambda(fg)_est}. According to the above estimates, the continuity of the map $\tau \mapsto S_0(\tau)\Lambda(fg):(0,\infty) \rightarrow X^{k+1}$  follows from the continuity of the kernel maps $D\tilde{G}_{\alpha(\cdot)},DG_{\alpha(\cdot)}:(0,\infty) \rightarrow L^1(\R^5)$, and $G_{\alpha(\cdot)},D^2 G_{\alpha(\cdot)}:(0,\infty) \rightarrow L^1(\R^7)$.
The estimate \eqref{S_0_Lambda_estimate_1} is obtained similarly.
\end{proof}
\begin{corollary}\label{Cor:Lambda_est}
	Given $f,g \in  X^4$ the function
	\begin{equation}\label{Def:Lambda_lambda_map}
		\tau \mapsto \Lambda [\phi \, S_0(\tau)\Lambda(fg)]
	\end{equation}
maps $(0,\infty)$ continuously into $X^4$. Furthermore, we have that
\begin{equation}\label{Eq:Lambda_S_0_Lambda_est}
	\|\Lambda [\phi \, S_0(\tau)\Lambda(fg)] \|_{X^{4}} \lesssim \beta(\tau)e^{-\frac{\tau}{4}}\|f \|_{X^4}\|g \|_{X^4}.
\end{equation}
for all $\tau>0$ and all $f,g \in  X^4$.
\end{corollary}
\begin{proof}
Continuity of the function \eqref{Def:Lambda_lambda_map} follows from \eqref{Eq:Lambda_est} and the continuity of the map $\tau \mapsto S_0(\tau) \Lambda (fg) : (0,\infty) \rightarrow X^5$.
The estimate \eqref{Eq:Lambda_S_0_Lambda_est} follows from \eqref{Eq:Lambda_est} and \eqref{Eq:S_0_Lambda_est}.
\end{proof}
 	Now we propagate the smoothing estimates of $S_0(\tau)$ to $S(\tau)$. We take the perturbative approach, and for that we need the following result.
\begin{lemma}\label{Lem:Duhamel}
	Let $f \in C^\infty_{c,\emph{rad}}(\R^5)$, $k \geq 3$, and $\tau \geq 0$. Then the following relations hold in $X^k$ 
	\begin{equation}\label{Eq:Duhamel_the_wrong_one}
		S(\tau)f = S_0(\tau)f + \int_{0}^{\tau}S(\tau-s)L'S_0(s)f ds,
	\end{equation}
	\begin{equation}\label{Eq:Duhamel_the_right_one}
		S(\tau)f = S_0(\tau)f + \int_{0}^{\tau}S_0(\tau-s)L'S(s)f ds.
	\end{equation}
\end{lemma}
\begin{proof}
 Define $\xi_f:[0,\tau] \rightarrow X^k$ by $s \mapsto S(\tau-s)S_0(s)f$.  We prove that $\xi_f$ is continuously differentiable. More precisely, we show that
	\begin{equation}\label{Eq:Xi_prime}
		\xi_f'(s)= -S(\tau-s)LS_0(s)f+S(\tau-s)L_0 S_0(s)f=- S(\tau-s) L' S_0(s)f,
	\end{equation}
	which is a continuous function from $[0,\tau]$ into $X^k$. To show this, we first write
	\begin{align*}
		\frac{\xi_f(s+h)-\xi_f(s)}{h} = \frac{S(\tau-s-h)-S(\tau-s)}{h}&S_0(s)f \\& \hspace{-2cm}+ S(\tau-s-h)\frac{S_0(s+h)f-S_0(s)f}{h},
	\end{align*}
and then by letting $h \rightarrow 0$ we get \eqref{Eq:Xi_prime}. For the first term above, this follows from the fact that $S_0(s)f \in \mc C \subset \mc D(\mc L_k)$ and that $\mc L_kf=Lf$ for $f \in \mc C.$ The conclusion for the second term follows by similar reasoning for $S_0(\tau)$, together with the strong continuity of $S(\tau)$ in $X^k$.
Now, continuity of $\xi'_f$ follows from the continuity of the map 
\begin{equation}\label{Def:L'_map}
	s \mapsto  L' S_0(s)f :[0,\tau] \rightarrow X^k
\end{equation}
 and the strong continuity of $S(\tau)$ in $X^k$. We note that, according to the definition of $L'$, the continuity of \eqref{Def:L'_map} follows from the strong continuity of $S_0(\tau)$ on $X^{k+1}$ and the estimate \eqref{Eq:Lambda_est}.
    Finally, by integrating \eqref{Eq:Xi_prime}, we get \eqref{Eq:Duhamel_the_wrong_one}. 
To prove \eqref{Eq:Duhamel_the_right_one}, we do the analogous thing. Namely, we consider the function
	\begin{equation*}
		 s \mapsto \eta_f(s):=S_0(\tau-s)S(s)f : [0,\tau] \rightarrow X^k,
	\end{equation*}
	which is also continuously differentiable, with 
	\begin{equation*}
		\eta'_f(s)= S_0(\tau-s)LS(s)f - S_0(\tau-s)L_0S(s)f = S_0(\tau-s) L' S(s)f.
	\end{equation*}
	To establish differentiability, it is important to note that according to the definition of the operator domain, by Lemma \ref{Lem:Lambda(fg)_est} we have that that $\mc D(\mc L_{k+1}) \subset \mc D(\mc L_{0,k})$, and therefore $S(s)f \in \mc D(\mc L_{0,k})$ for every $k \geq 3$. Continuity of $\eta'_f$, similarly to above, follows from the continuity of $s \mapsto  L' S(s)f :[0,\tau] \rightarrow X^k$ and the strong continuity of $S_0(\tau)$ in $X^k$.
\end{proof}
 Recall the operator $N$ from \eqref{Def:Operators}. According to Lemma \ref{Lem:Lambda(fg)_est} we have that $N : X^{k} \rightarrow X^{k-1}$ for $k \geq 4$. Also, recall the projection operator $\mc P=\mc P_{X^k}$ from \eqref{Def:P_k}. Now we prove the central result of this section.
\begin{proposition}\label{Prop:Nonlin_est_S}
	If $f \in X^4$ then
	\begin{equation}\label{Def:tau_S_map}
		\tau \mapsto S(\tau)N(f):(0,\infty) \rightarrow X^4
	\end{equation} 
is a continuous map. Furthermore, there exist $\tilde{\omega},\omega>0$ such that
	\begin{align}
			\| S(\tau)[N(f)-N(g)] \|_{X^{4}} &\lesssim e^{\tilde{\omega}\tau}(\| f \|_{X^4} + \| g \|_{X^4}) \| f-g \|_{X^4}, \label{Eq:S(t)N} \\
		\| (1-\mc P)S(\tau)[N(f)-N(g)] \|_{X^{4}} &\lesssim e^{-\omega\tau}(\| f \|_{X^4} + \| g \|_{X^4}) \| f-g \|_{X^4},  \label{Eq:(1-P)S_estimate}
	\end{align}
	for all $\tau > 0$ and all $f,g \in  X^4$. In addition, for every $k \geq 4$ there exists $\tilde{\omega}_k >0$  such that
	\begin{gather} \label{Eq:Smoothing_S}
		 \| S(\tau)[N(f)-N(g)] \|_{X^{k+1}}
		  \lesssim \beta(\tau) e^{\tilde{\omega}_k  \tau}(\| f \|_{X^k} + \| g \|_{X^k}) \| f-g \|_{X^k},
	\end{gather}
for all $\tau > 0$ and all $f,g \in  X^k$.
\end{proposition}	
\begin{proof}
As usual, it is enough to prove the proposition for $f,g \in C^\infty_{c,\text{rad}}(\R^5)$. We first establish the three estimates above, and then the continuity of the map \eqref{Def:tau_S_map}. The relations \eqref{Eq:S(t)N} and \eqref{Eq:(1-P)S_estimate} follow from \eqref{Eq:Duhamel_the_wrong_one}, \eqref{S_0_Lambda_estimate_1}, Theorem \ref{Thm:DecayX}, and Corollary \ref{Cor:Lambda_est}. 
 Here we explicitly show only \eqref{Eq:(1-P)S_estimate}. 
	We have that equation \eqref{Eq:Duhamel_the_wrong_one},  via \eqref{S_0_Lambda_estimate_1}, \eqref{Eq:S_decay}, and \eqref{Eq:Lambda_S_0_Lambda_est} implies
	\begin{align*}
		\| (1-\mc P)S(\tau)\Lambda(uv) \|_{X^4}
		&\lesssim \| S_0(\tau)\Lambda(uv) \|_{X^4} + \int_{0}^{\tau}\| (1-\mc P)S(\tau-s) L' S_0(s)\Lambda(uv)\|_{X^4} ds  \\ 
		 &\lesssim e^{-\frac{\tau}{4}}\| u\|_{X^4}\| v\|_{X^4} 
		 + e^{-{\omega}_4\tau}
		 \int_{0}^{\tau} e^{({\omega}_4-\frac{1}{4})s }\beta(s)ds \, \|u \|_{X^4}\|v \|_{X^4}\\
		&\lesssim (e^{-\frac{\tau}{4}} + e^{-{\omega}_4\tau} \tau)
		\| u\|_{X^4}\| v\|_{X^4} \\
		& \lesssim e^{-\frac{{\omega}_4}{2}\tau}
		\| u\|_{X^4}\| v\|_{X^4},
	\end{align*}
for all $\tau >0$ and $u,v \in C^\infty_{c,\text{rad}}(\R^5)$. The same type of estimate holds for $6uv$ instead of $\Lambda(uv)$, and \eqref{Eq:(1-P)S_estimate} then follows by letting $u=f+g$ and $v=f-g$. 
To get the estimate \eqref{Eq:Smoothing_S}, we first use \eqref{Eq:Duhamel_the_right_one} and Proposition \ref{Prop:S_0_est} to obtain
\begin{align*}
	\| S(\tau)\Lambda(uv) \|_{X^{k+1}}
	&\lesssim \| S_0(\tau)\Lambda(uv) \|_{X^{k+1}} + \int_{0}^{\tau}\| S_0(\tau-s) L' S(s)\Lambda(uv)\|_{X^{k+1}} ds  \\ 
	&\lesssim \beta(\tau)e^{-\frac{\tau}{4}}\| u\|_{X^k}\| v\|_{X^k} 
	+ e^{-\frac{\tau}{4}}
	\int_{0}^{\tau} e^{\frac{s}{4}} \| S(s)\Lambda(uv) \|_{X^{k+1}} ds,
\end{align*}
for all $\tau >0$ and $u,v \in C^\infty_{c,\text{rad}}(\R^5)$. Now, by Gronwall's lemma, from this estimate we get \eqref{Eq:Smoothing_S}.
Finally, continuity of the map \eqref{Def:tau_S_map} follows from \eqref{Eq:Duhamel_the_wrong_one}, according to Proposition \ref{Prop:S_0_est}, Corollary \ref{Cor:Lambda_est}, and the strong continuity of $S(\tau)$ in $X^4$.
\end{proof}
As the last result of this section, we prove the local Lipschitz continuity in $X^4$ of the composition of $\mc P$ and $N$.
\begin{lemma}
	We have that
\begin{equation}\label{Eq:Nonlin_est_P}
		\| \mc PN(f)- \mc PN(g) \|_{X^4} \lesssim (\| f \|_{X^4} + \| g \|_{X^4}) \| f-g \|_{X^4},
\end{equation}
for all $f,g \in X^4$.
\end{lemma}
\begin{proof}
	By definition, for $u,v \in X^4$ we have
	\begin{equation*}
		\mc P N(uv) = \langle N(uv),g\rangle_{\mc H} \, g.
	\end{equation*}
	Therefore, by Cauchy-Schwarz, the embedding $X^3 \hookrightarrow \mc H$, and Lemma \ref{Lem:Lambda(fg)_est}, we get that
	 \begin{equation*}
	 	\|\mc  PN(uv) \|_{X^4} \leq |\langle N(uv),g\rangle_{\mc H}| \|g \|_{X^4}
		\lesssim \| N(uv) \|_{X^3} 	 	
	 	\lesssim
	 	 \| u \|_{X^4}\| v \|_{X^4},
	 \end{equation*} 
for all $u,v \in X^4$.
The estimate \eqref{Eq:Nonlin_est_P} then follows by letting $u=f+g$ and $v=f-g$.
\end{proof}

\section{Construction of strong solutions}\label{Sec:Str_sol}
\noindent For simplicity, from now on we will drop the subscript in $\| \cdot \|_{X^4}$, and assume that an unspecified norm corresponds to $X^4$.
With the linear theory and the nonlinear estimates from the previous section at hand, we turn to constructing solutions to \eqref{Eq:Perturbed}. For convenience, we copy here the underlying Cauchy problem
\begin{equation}\label{Eq:Vector_pert_2}
	\begin{cases}
		~\partial_\tau \psi(\tau) =  \mc L\psi(\tau) +  N(\psi(\tau)),\\
		~\psi(0)= U( v, T).
	\end{cases}	
\end{equation}
To solve \eqref{Eq:Vector_pert_2}, we utilize the standard techniques from dynamical systems theory. First, we use the fact that $\mc L$ generates the semigroup $S(\tau)$, to rewrite \eqref{Eq:Vector_pert_2} into the integral form
\begin{equation}\label{Eq:Duhamel}
	\psi(\tau)= S(\tau) U(v,T) + \int_{0}^{\tau} S(\tau-s) N(\psi(s))ds.
\end{equation} 
Then, as $S(\tau)$ decays exponentially on the stable subspace, we employ a  fixed point argument to show existence of global solutions for small initial data. Obstruction to this is, of course, the presence of the linear instability $\la=1$. Nevertheless, as this eigenvalue is an artifact of the time translation symmetry, we use a Lyapunov-Perron type argument to suppress it by appropriately choosing the blowup time. 
Before stating the first result, we make some technical preparations. First, we introduce the Banach space
\begin{equation*}
	\mc X := \{ \psi \in C([0,\infty), X^4) : 
	\| \psi  \|_{\mc X} := \sup_{\tau>0}e^{\omega\tau}\|\psi(\tau)\| < \infty
	\}, 
\end{equation*}
where $\omega$ is from Proposition \ref{Prop:Nonlin_est_S}. Then, we denote
\begin{equation*}
	\mc X_\delta := \{ \psi \in \mc X: \| \psi \|_{\mc X} \leq \delta \}.
\end{equation*}
Now, we define a correction function $C: X^4 \times \mc X \rightarrow X^4$ by
\begin{equation}\label{Def:Correction_term}
	 C( u, \psi):= \mc P \left(  u + \int_{0}^{\infty}e^{- s} N(\psi(s))ds \right),
\end{equation}
and a map $ K_{u} : \mc X \rightarrow C([0,\infty),X^4)$ by
\begin{equation*}
	 K_{ u}( \psi)(\tau) := S(\tau)\big( u -  C( u, \psi) \big) + \int_{0}^{\tau} S(\tau-s) N(\psi(s))ds.
\end{equation*}
The fact that $K_u(\psi)(\tau)$ is a well-defined element of $X^4$ for every $\tau \geq 0$, follows from \eqref{Eq:S(t)N}. Then, the continuity of $K_u(\psi):[0,\infty) \rightarrow X^4$ follows from the continuity of $\psi$ and that of $\tau \mapsto S(\tau)N(f):(0,\infty) \rightarrow X^4$ given $f \in X^4$.
\begin{proposition}\label{Prop:K}
	For all  sufficiently small $\delta>0$ and all  sufficiently large $C > 0$ the following holds. If $u \in \mc B_{\delta/C}$ then there exits a unique $\psi = \psi({u}) \in \mc X_\delta$ for which
	\begin{equation}\label{Eq:K(u,psi)}
		\psi = { K}_{u}( \psi).
	\end{equation}
	Furthermore, the map ${u} \mapsto \psi({u}): \mc B_{\delta/C} \rightarrow \mc X$ is  Lipschitz continuous. 
\end{proposition}
\begin{proof}
	To utilize the decay of $ S(\tau)$ on the stable subspace, we write  $ K_{u}$ in the following way
	\begin{equation*}
		 K_{u}( \psi)(\tau)=  S(\tau)(1- \mc P) u +  \int_{0}^{\tau} S(\tau-s) (1- \mc P) N(\psi(s))ds - \int_{\tau}^{\infty} e^{\tau- s}  \mc P  N(\psi(s))ds.
	\end{equation*}
	Then, according to Proposition \ref{Prop:Nonlin_est_S} we get that if $\psi(s) \in \mc B_\delta$ for all $s \geq 0$ then
	\begin{equation*}\label{Eq:K(u,phi)_2}
		\|  K_{ u}( \psi)(\tau) \| \lesssim e^{-\omega \tau}\|  u \| + e^{-\omega \tau} \int_{0}^{\tau}e^{\omega s}\| \psi(s) \|^2ds + e^{\tau}\int_{\tau}^{\infty} e^{-s} \|  \psi(s)  \|^2 ds.
	\end{equation*}
	Furthermore, if $  u \in \mc B_{\delta/C}$ and $\psi \in \mc X_{\delta}$ then the above estimate implies the bound
	\begin{equation*}
		e^{\omega \tau}\|  K_{ u}( \psi)(\tau) \| \lesssim \tfrac{\delta}{C} + \delta^2 + \delta^2 e^{-\omega\tau}.
	\end{equation*}
	Also, we similarly get that
	\begin{equation*}
		e^{\omega \tau}\|  K_{ u}( \psi)(\tau) -  K_{ u}( \varphi)(\tau) \| \lesssim (\delta + \delta e^{-\omega \tau})\| \psi - \varphi \|_{\mc X}
	\end{equation*}
	for all $\psi,\varphi \in \mc X_\delta$.
	Now, the last two displayed equations imply that for all small enough $\delta$ and for all large enough $C$, given $ u \in \mc B_{\delta/C}$ the operator $ K_{ u}$ is contractive on $\mc X_\delta$, with the contraction constant $\frac{1}{2}$. Consequently, the existence and uniqueness of solutions to \eqref{Eq:K(u,psi)} follows from the Banach fixed point theorem. The show continuity of the map $ u \mapsto \psi( u) $ we utilize the contractivity of $ K_{ u}$. Namely, we have the estimate
	\begin{align*}
		\| \psi( u)(\tau) - \psi( v)(\tau) \| &=
		\|  K_{ u}( \psi( u))(\tau) -  K_{ v}( \psi( v))(\tau)\| \\
		&\leq  	\|  K_{ u}( \psi( u))(\tau) -  K_{ v}( \psi( u))(\tau)\| + \|  K_{ v}( \psi( u))(\tau) -  K_{ v} (\psi( v))(\tau)\|  \\
		&\leq \|  S(\tau)(1- \mc P)( u- v) \| + \tfrac{1}{2}\| \psi( u)(\tau) - \psi( v)(\tau) \| \\ 
		&\leq  Ce^{-\omega\tau} \|  u -  v \| + \tfrac{1}{2}\| \psi( u)(\tau) - \psi( v)(\tau) \|,
	\end{align*}
	wherefrom the Lipschitz continuity follows.
\end{proof}

\begin{lemma}\label{Lem:U(v,t)}
	For $\delta \in (0,\frac{1}{2}]$ and ${ v} \in X^4 $ the map
	\begin{equation*}
		T \mapsto { U}({ v},T):[1-\delta,1+\delta] \rightarrow X^4
	\end{equation*}
	is continuous. In addition, we have that
	\begin{equation}\label{Eq:U(v,T)_est}
		\| { U}({ v},T) \| \lesssim  \| { v} \| + |T-1|
	\end{equation}
	for all ${ v} \in X^4 $ and all $T \in [\frac{1}{2},\frac{3}{2}]$.

\end{lemma}
\begin{proof}
	Fix $\delta \in (0,\frac{1}{2}]$ and $ v \in X^4$. Then for $T,S \in [1-\delta,1+\delta]$ we have that
	\begin{equation}\label{Eq:T-S}
		U(v,T) -  U(v,S) = (T-S)w_0(\sqrt{T}\cdot) + S\big( w_0(\sqrt{T}\cdot) - w_0(\sqrt{S}\cdot) \big). 
	\end{equation}
	Let $\varepsilon >0$. By density, we know that there exists $\tilde{w}_0 \in C^\infty_{c,\text{rad}}(\R^5)$ for which $\| w_0 - \tilde{w}_0 \| < \varepsilon$. Now, by writing
	\begin{equation*}
		w_0(\sqrt{T}\cdot) - w_0(\sqrt{S}\cdot) = \big ( w_0(\sqrt{T}\cdot)-\tilde{w}_0(\sqrt{T}\cdot) \big) + \big( \tilde{w}_0(\sqrt{T}\cdot) - \tilde{w}_0(\sqrt{S}\cdot) \big) + \big( \tilde{w}_0(\sqrt{S}\cdot) - w_0(\sqrt{S}\cdot) \big)
	\end{equation*}
	and using the fact that 
	$\lim_{S \rightarrow T} \| \tilde{w}_0(\sqrt{T}\cdot) - \tilde{w}_0(\sqrt{S}\cdot)  \|=0,
	$
	from \eqref{Eq:T-S} we see that 
	\begin{equation*}
		\lim_{S \rightarrow T} \| U( v,T) - U( v,S) \| \lesssim \varepsilon.
	\end{equation*}
	Then, continuity follows by letting $\varepsilon \rightarrow 0$. For the second part of the lemma, we write  $ U( v,T)$ in the following way
	\begin{equation}\label{Eq:U(v,T)}
		 U( v,T) = Tv(\sqrt{T}\cdot)
+
			T \phi(\sqrt{T}\cdot)-\phi
	\end{equation}
 From here, the estimate \eqref{Eq:U(v,T)_est} follows.
\end{proof}

	Finally, by using the results above, we prove that given initial datum $v$ that is small in $X^4$, there exists a time $T$ and an exponentially decaying solution $\psi \in C([0,\infty),X^4)$ to \eqref{Eq:Duhamel}.

\begin{theorem}\label{Thm:CoMain}
	There  exist $\delta,N>0$ such that the following holds. If
	\begin{equation}\label{Eq:v_smallness}
		{ v} \in  X^4, \quad \text{{v} is  real-valued}, \quad \text{and} \quad \| { v} \| \leq \tfrac{\delta}{N^2},
	\end{equation}
	then there exist $T \in [1-\frac{\delta}{N}, 1+\frac{\delta}{N}]$ and $ \psi \in \mc X_{\delta}$ such that \eqref{Eq:Duhamel} holds for all $\tau \geq 0.$ 
\end{theorem}
\begin{proof}
	Lemma \ref{Lem:U(v,t)} and Proposition \ref{Prop:K} imply that for all small enough $\delta$ and all large enough $N$ we have that if $ v$ satisfies \eqref{Eq:v_smallness} and $T \in [1-\frac{\delta}{N} , 1+\frac{\delta}{N}]$ then there is a unique $\psi=\psi({ v},T) \in \mc X_\delta $ that solves
	\begin{equation}\label{Eq:Duhamel_C}
		\psi(\tau) = S(\tau)\big( U( v,T) -  C( U( v,T), \psi ) \big) + \int_{0}^{\tau} S(\tau-s) N\big(\psi (s)\big)ds.
	\end{equation}
	We remark that $\psi(\tau)$ is real-valued for all $\tau \geq 0$, since the set of real-valued functions in $X^4$ is invariant under the action of both $ S(\tau)$ and $\mc P$. Now, to construct solutions to \eqref{Eq:Duhamel}, we prove that there is a choice of $\delta$ and $N$ such that for any $ v$ that satisfies \eqref{Eq:v_smallness} there is $T=T( v) \in[1-\frac{\delta}{N} , 1+\frac{\delta}{N}]$ for which
	the correction term in \eqref{Eq:Duhamel_C} vanishes.
	As $ C$ takes values in $\rg  \mc P = \langle  g \rangle$, it is enough to show existence of $T$ for which
	\begin{equation}\label{Eq:Fixed_pt}
		\big\langle C( U( v,T), \psi({ v},T)),  g\big\rangle_{X^4}=0.
	\end{equation} 
	We therefore consider the real function $T \mapsto \langle C( U( v,T), \psi({ v},T)),  g \rangle_{X^4} $ and employ the Brouwer fixed point theorem to prove that it vanishes on $[1-\frac{\delta}{N} , 1+\frac{\delta}{N}]$. The central observation to this end is that, according to \eqref{Def:nu}, we have
	\begin{equation*}
		\partial_T \,
			T \phi (\sqrt{T}\cdot)  \big|_{T=1} 
		=    c g,
	\end{equation*}
	for some $c>0$. Based on this, by Taylor's formula, from \eqref{Eq:U(v,T)}  we get that 
	\begin{equation*}
		\big\langle \mc P U( v,T) ,  g \big\rangle_{X^4}=  c \|  g\|^2 \,(T-1)+R_1( v,T),
	\end{equation*}
	where  $R_1( v,T)$ is continuous in $T$ and $R_1( v,T) \lesssim \delta/N^2$.
	Furthermore, based on the definition of the correction function $ C$, we similarly conclude that
	\begin{equation*}
		\big\langle C( U( v,T), \psi({ v},T)),  g \big\rangle_{X^4}= c \|  g\|^2 \,(T-1) +  R_2( v , T),
	\end{equation*}
	where $T \mapsto R_2( v,T)$ is a continuous, real-valued function on $[1-\frac{\delta}{N} , 1+\frac{\delta}{N}]$, for which  $R_2( v,T) \lesssim \delta/N^2 + \delta^2.$
	Therefore, there is a choice of sufficiently large $N$ and sufficiently small $\delta$ such that $|R_2( v,T)| \leq c  \| g \|^2 \frac{\delta}{N}$. Based on this, we get that \eqref{Eq:Fixed_pt}  is equivalent to
	\begin{equation}\label{Eq:T}
		T=F(T)
	\end{equation}
	for some function $F$ which maps the interval $[1-\frac{\delta}{N} , 1+\frac{\delta}{N}]$ continuously into itself. Consequently, by the Brouwer fixed point theorem we infer the existence of $T \in [1-\frac{\delta}{N} , 1+\frac{\delta}{N}]$ for which \eqref{Eq:T}, and therefore \eqref{Eq:Fixed_pt}, holds. The claim of the theorem follows.
\end{proof}

\section{Upgrade to classical solutions}\label{Sec:Upgrade_to_class}
\noindent  In this section we show that if the initial datum $v$ is smooth and rapidly decaying, then the corresponding strong solution to \eqref{Eq:Duhamel} is in fact smooth, and satisfies \eqref{Eq:Vector_pert_2} classically. To accomplish this, we first use abstract results of the semigroup theory to upgrade  strong solutions to classical ones in the semigroup sense. Then we use repeated differentiation together with Schwarz's theorem on mixed partials to upgrade these to smooth solutions that solve \eqref{Eq:Vector_pert_2} classically.

\begin{proposition}\label{Prop:Upgrade_to_class}
	If $v$ from Theorem \ref{Thm:CoMain} belongs to the radial Schwartz class $\mc S_{\emph{rad}}(\R^5)$, then the function $\Psi(\tau,\xi):=\psi(\tau)(\xi)$ belongs to $C^\infty([0,\infty)\times \R^5)$  and satisfies
	\begin{equation}\label{Eq:Classical_sim_var}
		\partial_\tau \Psi(\tau,\cdot) = L \Psi(\tau,\cdot) + N(\Psi(\tau,\cdot)) 
	\end{equation}
	 in the classical sense.
\end{proposition}
\begin{proof}
	Recall from the linear theory that 
	\begin{equation}
		 S(\tau)|_{X^k} =  S_k(\tau).
	\end{equation}
	By this, from \eqref{Eq:Duhamel} and \eqref{Eq:Smoothing_S} we have that there is $\alpha \in \R$ such that 
	\begin{align*}
		\|\psi(\tau)\|_{X^{5}} &\lesssim 
		e^{\alpha \tau} \| U( v,T)\|_{X^{5}} + \int_{0}^{\tau} \| S(\tau-s) N(\psi(s))\|_{X^{5}}ds\\
		& \lesssim e^{\alpha \tau} \| U( v,T)\|_{X^{5}} + \int_{0}^{\tau}\beta(\tau-s) e^{\tilde{\omega}_4 (\tau-s)} \| \psi(s)\|^2_{X^{4}}ds.
	\end{align*}
	Consequently, $\psi(\tau) \in X_{5}$ for all $\tau \geq 0$. Since $ U(v,T) \in X^k$ for all $ k \geq 5$, we proceed inductively to get that $\psi(\tau) \in X^k$ for every $k \geq 5$. Then, by the embedding $X^k \hookrightarrow C^{k-3}_{\text{rad}}(\R^5)$ we conclude that $\psi(\tau) \in C^\infty(\R^5)$ for all $\tau \geq 0$.
	
	To establish regularity in $\tau$, we do the following. First, according to \eqref{Eq:S(t)N}, from \eqref{Eq:Duhamel} by Gronwall's inequality we conclude that $\psi :[0,\mc T] \rightarrow X^4$ is Lipschitz continuous for every $\mc T>0$.
	Consequently, according to Lemma \eqref{Lem:Lambda(fg)_est} we have that $\tau \mapsto N(\psi(\tau)):[0,\mc T] \rightarrow X^3$ is Lipschitz continuous for every $\mc T >0$. This, together with the fact that $U(v,T) \in \mc D( \mc L_3)$ implies that $\psi \in C^1([0,\infty),X^3)$, and $\psi$ satisfies \eqref{Eq:Vector_pert_2} in $X^3$ in the operator sense (see, e.g.,  Cazenave-Haraux \cite{CazHar98}, p.~51, Proposition 4.1.6, (ii)). Furthermore, as $X^3$ is continuously embedded in $L^\infty(\R^5)$ the $\tau$-derivative holds pointwise. Consequently, by (a strong version of) the Schwarz theorem (see, e.g., Rudin \cite{Rud76}, p.~235, Theorem 9.41), we  conclude that mixed derivatives of all orders in $\tau$ and $\xi$  exist, and we thereby infer smoothness  of $(\tau,\xi) \mapsto \psi(\tau)(\xi)$.
\end{proof}

\begin{proof}[Proof of Theorem \ref{Thm:Main}]
	Due to Lemma \ref{Lem:Sobolev_equiv} we can choose $\varepsilon>0$  small enough such that
	$$
	\| \varphi_0 \|_{H^{3}(\R^3)}   < \varepsilon \quad  \text{implies} \quad 	
	\| v \|_{\dot{H}^1 \cap \dot{H}^{4}(\R^5)} < \frac{\delta}{N^2},
	$$
	for $\delta,N$ from Theorem \ref{Thm:CoMain}.  Then, according to Theorem \ref{Thm:CoMain} there exists a solution $\psi \in C([0,\infty),X^4)$ to \eqref{Eq:Duhamel}, for which
	\begin{equation}\label{Eq:Est_Phi}
		\|\psi(\tau)\|_{X^4} \leq \delta e^{-\omega \tau}.
	\end{equation}
	Now, since, by assumption,  $v$ belongs to $\mc S_{\text{rad}}(\R^5)$, Proposition \ref{Prop:Upgrade_to_class} implies that $\Psi(\tau,\xi) = \phi(\xi) + \psi(\tau)(\xi)$ is smooth and solves \eqref{Eq:NLH_sim_var} classically. Therefore,
	\begin{equation*}
		w(t,x)=\frac{1}{T-t}{\Psi}\left(\log\left(\frac{T}{T-t}\right),\frac{x}{T-t}\right)
	\end{equation*}
	belongs to $C^\infty([0,T)\times \R^5)$ and solves the system \eqref{Eq:NLH} on $[0,T)\times \R^5$ classically. This then yields a smooth solution to \eqref{Eq:KS}
	\begin{equation*}
		u(t,x)=\frac{1}{T-t}\left[ \Phi\left(\frac{x}{\sqrt{T-t}}\right) +  \varphi \left(t, \frac{x}{\sqrt{T-t}}\right) \right],
	\end{equation*}
	where, according to \eqref{Eq:Sobol_equiv} and \eqref{Eq:Est_Phi} we have
	\begin{align*}
	\| \varphi(t,\cdot) \|_{H^3(\R^3)} \simeq	\| \varphi(t,\cdot) \|_{L^2 \cap \dot{H}^3(\R^3)} \simeq \| \psi(-\log(T-t)-\log T) \|_{\dot{H}^1 \cap \dot{H}^4(\R^5)} \lesssim (T-t)^\omega, 
	\end{align*}
	as $t \rightarrow T^-$.
\end{proof}

\appendix

\section{Estimates of local Sobolev norms}

\begin{lemma}\label{Lem:Local_Sobolev}
	Let $k \in \N$ and $R>0.$ Then 
	
	\begin{equation}\label{Eq:Local_Sob_est}
		\| \partial^\alpha u \|_{L^2(\B^5_R)} \lesssim \sum_{j=0}^{k} \| D^j u \|_{L^2(\B^5_R)}
	\end{equation}
for all $u \in C^\infty_{\emph{rad}}(\B^5_R)$ and all $\alpha \in \N_0^5$ with $|\alpha|=k$.
\end{lemma}

\begin{proof}
	We prove the claim for $R=1$ as the general case follows by scaling. Let $\chi : \R^5 \rightarrow [0,1]$ be a smooth radial function such that $\chi(x) =0$ for $|x|\leq \frac{5}{4}$ and $\chi(x) = 0$ for $|x| \geq \frac{3}{2}$. We then define for $u=\tilde{u}(|\cdot|) \in C_{\text{rad}}^k(\B^5)$ the extension operator
	\begin{equation*}
		\tilde{E}u :=
		\begin{cases}
			u(x) , \quad & x \in \B^5, \\
			-\tilde{u}(2-|x|) + 2 \displaystyle{\sum_{j=0}^{\frac{k-1}{2}}} \dfrac{ \tilde{u}^{(2j)}(1)}{(2j)!}(|x|-1)^{2j} , & x \in \B^5_2 \setminus \B^5, ~k \text{ odd},\\
			 \tilde{u}(2-|x|) + 2 \displaystyle{\sum_{j=1}^{\frac k2}} \dfrac{ \tilde{u}^{(2j-1)}(1)}{(2j-1)!}(|x|-1)^{2j-1} , & x \in \B^5_2 \setminus \B^5, ~k \text{ even},\\
			0, & x \in \R^5 \setminus \B^5_2,
		\end{cases}
	\end{equation*}
and then by means of the cut-off $\chi$ we let
\begin{equation}\label{Eq:E}
	Eu:= \chi \tilde{E}u.
\end{equation}
Note that $E: C^k_{\text{rad}}(\B^5) \rightarrow C^k_{c,\text{rad}}(\R^5)$.
By denoting with $\tilde{D}^iu$ the radial profile of $D^i u$ we have that
\begin{equation}\label{Eq:uD}
	|\tilde{u}^{(j)}(1)| \lesssim \sum_{i=0}^{j} |\tilde{D}^iu(1)|.
\end{equation}
Furthermore, by the fundamental theorem of calculus,
\begin{align}
	|\tilde{D}^{2i+1}u(1)| &\lesssim \left|\int_{0}^{1}\partial_r(r^4 \tilde{D}^{2i+1}u(r))dr \right| \nonumber \\
	& \lesssim \left(\int_{0}^{1}\left| r^{-4} \partial_r(r^4\tilde{D}^{2i+1}u(r)) \right|^2r^4 dr \right)^{\frac{1}{2}} \lesssim \| D^{2i+2}u \|_{L^2(\B^5)}, \label{Eq:D^{2i+1}}
\end{align}
and by Hardy's inequality (see, e.g., \cite{Glo22}, Lemma 2.12) we infer that 
\begin{equation}\label{Eq:D^{2i}}
	|\tilde{D}^{2i}u(1)| \lesssim  \| D^{2i}u \|_{L^2(\B^5)} + \| D^{2i+1}u \|_{L^2(\B^5)}.
\end{equation}
Therefore, from \eqref{Eq:uD}, \eqref{Eq:D^{2i+1}} and \eqref{Eq:D^{2i}} we have that
\begin{equation}\label{Eq:uD2}
	|\tilde{u}^{(j)}(1)| \lesssim \sum_{i=1}^{j+1} \|D^iu \|_{L^2(\B^5)}.
\end{equation}
Now, based on these results, from \eqref{Eq:E} we get that for $i \leq k$
\begin{equation*}
	\| D^k Eu \|_{L^2(\R^5 \setminus \B^5)} =  \| D^k Eu \|_{L^2(\B^5_{3/2} \setminus \B^5)} \lesssim \sum_{j=0}^{k} \| D^iu \|_{L^2(\B^5)}.
\end{equation*}
Therefore, we finally infer that
\begin{align*}
	\|\partial^\alpha u \|_{L^2(\B^5)} & \lesssim \| \partial^\alpha Eu \|_{L^2(\R^5)} \lesssim \| \mc F[\partial^\alpha Eu] \|_{L^2(\R^5)}
	\lesssim \| |\cdot|^k Eu \|_{L^2(\R^5)} \\ &\lesssim \| D^k Eu \|_{L^2(\R^5)}	
	\lesssim  \| D^k u  \|_{L^2(\B^5)} + \| D^k u  \|_{L^2(\R^5 \setminus \B^5)} 
	 \lesssim \sum_{j=0}^{k} \| D^j u \|_{L^2(\B^5)},
\end{align*}
for all $u \in C^\infty_{\text{rad}}(\B^5)$.
\end{proof}

\section{Equivalence of Sobolev norms for the reduced mass}

\begin{lemma}\label{Lem:Sobolev_equiv}
	Let $d \in \mathbb{N}$. For every $u=\tilde{u}(|\cdot|) \in C^\infty_{c,\emph{rad}}(\R^d)$ define $w=\tilde{w}(|\cdot|)$ by
	\begin{equation*}
			\tilde{w}(r):={r^{-d}} \int_{0}^{r}\tilde u(s)s^{d-1}ds.
	\end{equation*}
Then given $k \in \mathbb{N}_0$ we have that
\begin{equation}\label{Eq:Sobol_equiv}
	\| u \|_{\dot{H}^k(\R^d)} \simeq \| w  \|_{\dot{H}^{k+1}(\R^{d+2})} 
\end{equation}
for all $u \in C^\infty_{c,\emph{rad}}(\R^d)$.
\end{lemma}

\begin{proof}
	The proof relies on the Bessel function representation of the Fourier transform of radial functions. Recall our convention \eqref{Def:FourierT}.
	Then, for a radial Schwartz function $f=\tilde{f}(|\cdot|)$, we have that
	\begin{align}\label{Eq:RadialFourier}
		\mathcal{F}_df(\xi)=|\xi|^{1-\frac{d}{2}}\int_{0}^{\infty}\tilde{f}(r)J_{\frac{d}{2}-1}(r|\xi|)r^{\frac{d}{2}}dr,
	\end{align}
	(see, e.g., Grafakos \cite{Gra08}, p.~429). Now,  for $\rho>0$ by partial integration we have
	\begin{align}\label{Eq:Fourier_bessel}
		\int_{0}^{\infty}\tilde{u}(r)J_{\frac{d}{2}-1}(r\rho)r^{\frac{d}{2}}dr \nonumber
		&= \int_{0}^{\infty}\big(\tilde{w}(r)r^d\big)'J_{\frac{d}{2}-1}(r\rho)r^{1-\frac{d}{2}}dr \\ \nonumber
		&= \rho\int_{0}^{\infty}\tilde{w}(r)\left(-J'_{\frac{d}{2}-1}(r\rho)+\frac{\frac{d}{2}-1}{r\rho}J_{\frac{d}{2}-1}(r\rho)\right)r^{\frac{d}{2}+1}dr \\ 
		&=
		 \rho\int_{0}^{\infty}\tilde{w}(r)J_{\frac{d}{2}}(r\rho)r^{\frac{d}{2}+1}dr,	
	\end{align}
where we used the recurrence relation for $J$-Bessel functions
\[
J_{\nu+1}(z)=-J'_{\nu}(z)+\frac{\nu}{z}J_{\nu}(z).
\]
According to \eqref{Eq:Fourier_bessel} and \eqref{Eq:RadialFourier} we have that
\begin{equation*}
	\| u \|_{\dot{H}^k(\R^d)} = \| |\cdot|^k \mc F_du \|_{L^2(\R^d)} \simeq \| |\cdot|^{k+1} \mc F_{d+2}w \|_{L^2(\R^{d+2})} = \| w \|_{\dot{H}^{k+1}(\R^{d+2})}.
\end{equation*}
\end{proof}


\section{An ODE result}\label{Sec:App_ODE}

\noindent  In this section, we use the following notation
\[ C_{e}^{\infty}[0,\infty):= \{ u \in C^{\infty}[0,\infty): u^{(2k+1)}(0) = 0, k \in \N_0 \}, \]
and note  $f \in  C_{\mathrm{rad	}}^{\infty}(\R^d)$ if and only if $f = \tilde f(|\cdot|)$ with  $\tilde f \in C_{e}^{\infty}[0,\infty)$. For $\rho \in [0,\infty)$ we set 
\[ V_0(\rho) := 2 \rho \tilde \phi(\rho), \quad V_1(\rho) := 2\rho \tilde  \phi'(\rho) + 12 \tilde  \phi(\rho), \]
where $\tilde \phi(\rho) = \frac{2}{2+\rho^2}$. Also, we let $\bar{\omega}_k$ denote the constant in Eq.~\eqref{Inequ:DissEst1}.

\begin{lemma}\label{Ap:ODE}  Let $k \in \N$, $k \geq 3$. Let $f$ be an element of  $C_{e}^{\infty}[0,\infty)$ with bounded support, and let $\lambda > \max \{2,\bar{\omega}_k \}$. Then there exists a function $u \in C^1[0,\infty) \cap C^{\infty}(0,\infty)$ which solves the equation
\begin{align}\label{Eq:Inhom_ODE}
 u''(\rho) + \left(\frac{4}{\rho} - \frac{1}{2} \rho  + V_0(\rho) \right ) u'(\rho) +  \left( V_1(\rho) - (\lambda + 1) \right) u(\rho) =  f(\rho)
\end{align}
on the interval $(0,\infty)$, satisfies $u'(0) = 0$, and given $j \in \N_0$ obeys the estimate
\begin{equation*}
	|u^{(j)}(\rho)| \lesssim \rho^{-2-2\lambda - j}
\end{equation*} as $\rho \to \infty$.
\end{lemma}

\begin{proof}
First, we construct a fundamental system for the homogeneous equation
\begin{align}\label{Eq:Hom_ODE}
 u''(\rho) + \left(\frac{4}{\rho} - \frac{1}{2} \rho  + V_0(\rho) \right ) u'(\rho) +  \left( V_1(\rho) - (\lambda + 1) \right) u(\rho) = 0.
\end{align}
We note that the origin $\rho = 0$ is a regular singular point. Hence, by the Frobenius method, there is a fundamental system  $\{ u_0, u_1 \}$ on $(0,\infty)$, where $u_0$ is analytic at $\rho=0$ with $u_0(0) = 1, u_0'(0)=0$, and $u_1(\rho) \sim \rho^{-3}$ near $\rho =0$. To analyze the behavior of solutions at infinity we write the equation in normal form.
With $\omega(r) := e^{\frac{r^2}{2}} r^{-2} (1+ 2 r^2)^{-1}$ and $v(r)\omega(r) =  u(2r)$, Eq.~\eqref{Eq:Hom_ODE} transforms into 

\begin{align}\label{normalform}
v''(r) - (r^2 + \mu) v(r) + V(r) v(r) = 0 
\end{align}
with $\mu = 4\lambda - 5 > 0$ and 
\[ V(r) = \frac{16}{(1+2r^2)^2} + \frac{8}{1+2r^2} - \frac{2}{r^2}. \]
By transforming the solutions of Eq.~\eqref{Eq:Hom_ODE} we obtain a fundamental system $\{v_0,v_1 \}$  for Eq.~\eqref{normalform} with $v_0(r) \sim r^2$ and $v_1(r) \sim r^{-1}$ for $r \to 0^+$.

For large values of the argument, the situation is more involved. For $r \geq 1$ and $V =0$, a fundamental system can be given in terms of parabolic cylinder functions $\{ U(\frac{\mu}{2}, \sqrt{2} \cdot), V(\frac{\mu}{2}, \sqrt{2} \cdot) \}$, with asymptotic behavior
\begin{align*}
  U(\tfrac{\mu}{2}, \sqrt{2} r) \sim e^{-\frac{1}{2} r^2} r^{- \frac{1}{2} (\mu+1)}, \quad V(\tfrac{\mu}{2}, \sqrt{2} r) \sim e^{\frac{1}{2} r^2} r^{\frac{1}{2} (\mu-1)},
\end{align*}
for $r \to \infty$, see for example \cite{NIST}. Our goal is to construct perturbatively a solution to Eq.~\eqref{normalform}, linearly independent of $v_0$, that behaves like $U(\tfrac{\mu}{2}, \sqrt{2} \cdot)$ at infinity. We make this fully explicit by considering a slightly different `free' equation first, namely
\begin{align*}
v''(r) - (r^2 + \mu) v(r) +  Q_{\mu}(r) v(r) = 0 
\end{align*}
with potential
\[
Q_{\mu}(r):= \mu^{-1} q( \mu^{-1/2} r), \quad q(r) = \frac{2-3r^2}{4(1+r^2)^2}.
\]
This equation has an explicit fundamental system  (see \cite{DonSch19}, Section 4.1.1),
\begin{align}\label{Fundamental_System_v0}
 v^{\pm}(r) = \tfrac{1}{\sqrt{2}}  \mu^{-\frac{1}{4}} (1 + \tfrac{r^2}{\mu} )^{-\frac14}  e^{\pm \mu \xi(\mu^{-1/2} r)}
\end{align}
with $ \xi(r) = \frac{1}{2} \log ( r + \sqrt{1+r^2}) + \frac{1}{2}r \sqrt{1+r^2}$
and Wronskian $W(v^{-},v^{+}) = 1$.  Note that 
\[\mu \xi(\mu^{-1} r)  = \tfrac{r^2}{2} + \tfrac{\mu}{2} \log(r)  + c_{\mu} + \varphi_{\mu}(r) \]
with $c_{\mu} \in \R$ and $\varphi_{\mu}(r) = \mc O(r^{-2})$ for $r \to \infty$, hence
\[ v^{-}(r) \sim e^{-\frac{1}{2} r^2} r^{- \frac{1}{2} (\mu+1)}, \quad v^{+}(r) \sim e^{\frac{1}{2} r^2} r^{\frac{1}{2} (\mu-1)}, \]
for $r \to \infty$. We add $Q_{\mu}$ to both sides of Eq.~\eqref{normalform} and put the potential  $V$ to the right hand side to obtain 
\begin{align}\label{ODE_Pert}
v''(r) - (r^2 + \mu) v(r) +  Q_{\mu}(r) v(r) =\mc O(r^{-2})  v(r).
\end{align}
Assuming $r \geq 1$, we show by a perturbative argument the existence of a solution $v_{\infty}$ to \eqref{ODE_Pert} 
which behaves like  $ v^{-}$ at infinity. For this, we set up a Volterra iteration by
reformulating Eq.~\eqref{ODE_Pert} as an integral equation using the variation of constants formula. More precisely, we look for a solution $v_{\infty}$ that satisfies 
\begin{align*}
v_{\infty}(r)  = v^{-}(r) + v^{+}(r) \int_{r}^{\infty}  v^{-}(s) \mc O(s^{-2})  v_{\infty}(s) ds - v^{-}(r) \int_{r}^{\infty}  v^{+}(s) \mc O(s^{-2})  v_{\infty}(s) ds.
\end{align*}
Noting that $v^{-}(r) > 0$ for all $r >0$, we set $h(r) := \frac{ v_{\infty}(r)}{  v^{-}(r)}$ and write the above equation as 
\begin{align}\label{Eq_hmin}
h(r) = 1 + \int_{r}^{\infty} K(r,s) h(s) ds 
\end{align}
where 
\begin{align*}
K(r,s) :=\left [  \frac{v^{+}(r)}{v^{-}(r)}  v^{-}(s)^2 - v^{+}(s) v^{-}(s) \right ] \mc O(s^{-2}).
\end{align*}
Explicitly,
\[ K(r,s) = \tfrac{1}{2}  \mu^{-\frac12} (1 + \tfrac{s^2}{4})^{-\frac12} 
\mc O(s^{-2})  \left ( e^{- 2 \mu (\xi(\mu^{-\frac12} s)  -  \xi(\mu^{-\frac12} r))} - 1 \right ). \] 
Using the fact that $\xi$ is monotonically increasing, we obtain the bound 
\[ |K(r,s)|  \lesssim s^{-3}, \] 
for $1 \leq r \leq s$. Thus,
\[ \int_{1}^{\infty} \sup_{r \in [1,s]} |K(r,s)| ds  \lesssim 1 \]
and we can apply standard results on Volterra equations (see, e.g., \cite{SchSoffStaub2010}, Lemma $2.4$) which yield the existence of a solution $h$ on $[1,\infty)$ with $|h(r)| \lesssim 1$  and 
\[ |h(r)-1| \lesssim \int_{r}^{\infty} |K(r,s)| ds \lesssim r^{-2}. \]
 By inspection (see also Remark $4.4$ in \cite{DonSch14a}), one finds that 
\begin{align}\label{Eq:SymbBehWeberhp}
|\partial_r^{k} (h(r)-1)|\lesssim_{k}  r^{-2-k}
\end{align}
for all $k \in \N$. This yields a  smooth solution 
\begin{align}\label{Fundamental_System_v}
v_{\infty}(r) =  v^{-}(r) [ 1 + \mc O(r^{-2}) ] 
\end{align}
to Eq.~\eqref{normalform} on $[1,\infty)$, where  the error term behaves like a symbol under differentiation. 

Now, by linearity, we have the representation
\begin{align}\label{Eq:connection}
v_{\infty} = c_0 v_0 + c_1 v_1,
\end{align}
for some constants $c_0,c_1 \in \C$. Suppose that $c_1 = 0$, i.e., $v_{\infty}$ and $v_0$ are linearly dependent. By transforming back, we would obtain a function $u \in C_{e}^{\infty}[0,\infty)$ with $u(\rho) = \mc O(\rho^{-2 - 2\lambda})$ as $\rho \to \infty$. In particular, $u(|\cdot|)$ would belong to $\mc C$ and satisfy
$(\lambda - \mc L_k)u(|\cdot|) = 0$ for some $\lambda > \bar{\omega}_k$. This, however, contradicts Eq.~\eqref{Inequ:DissEst1} stated in the proof of Proposition \ref{Prop:Semigroup_Sk}. 
We conclude that $\{ v_{\infty}, v_0\}$ is a fundamental system for Eq.~\eqref{normalform} on $(0,\infty)$, and we denote by $W := W(v_{\infty},v_0)(1)$ its Wronskian.

Now we turn to the inhomogeneous equation \eqref{Eq:Inhom_ODE}, which transforms into 
\begin{align}\label{normalform_inhom}
v''(r) - (r^2 + \mu) v(r) + V(r) v(r) = w(r)^{-1}f(r/2).
\end{align}
By the variation of constants formula we find a particular solution 
\begin{align*}
v(r) = \frac{v_0(r)}{W} \int_{r}^{\infty} & v_{\infty}(s) e^{-s^2/2} s^2 (1 + 2 s^2) f(\tfrac{s}{2}) ds \\
& + \frac{v_{\infty}(r)}{W} \int_0^{r} v_0(s) e^{-s^2/2} s^2 (1 + 2 s^2) f(\tfrac{s}{2}) ds.
\end{align*}
Obviously, $v \in C^{\infty}(0,\infty)$. Since $f$ has bounded support, the first integral vanishes for large $r$ and therefore there is a constant $c$ such that $v(r)= cv_{\infty}(r)$ for all large enough $r$. For $r \to 0$, the first integral converges, hence the behavior of the first term is governed by $v_0$. The second integral is of order  $\mc O(r^5)$ which compensates the singular behavior of $v_{\infty}$ at the origin. In particular, there is a constant $C$ such that $r^{-2}v(r)  \rightarrow C$ and $r^{-1}v'(r) \rightarrow 2C$ when $r \to 0^{+}$. By transforming back, we obtain a solution $u \in C^1[0,\infty) \cap C^{\infty}(0,\infty)$. By inspection, $u'(0) = 0$ and $u^{(k)}(\rho) = \mc O(\rho^{-2 - 2\lambda - k})$ for $\rho \to \infty$ and $k \in \N_0$. 
\end{proof}

	\bibliography{refs-KS}
	\bibliographystyle{plain}
	
\end{document}